\documentclass{ro}
\usepackage[colorlinks=true,citecolor=blue]{hyperref}
\usepackage{mathptmx, amsmath, amssymb, amsfonts, amsthm, mathptmx, enumerate, color,mathrsfs}
\usepackage{subfigure}
\usepackage{graphicx}
\usepackage{float}

\usepackage{epstopdf}
\usepackage{enumitem}

\newtheorem{theorem}{Theorem}[section]
\newtheorem{lemma}[theorem]{Lemma}
\newtheorem{proposition}[theorem]{Proposition}

\theoremstyle{definition}
\newtheorem{definition}[theorem]{Definition}

\newtheorem{example}[theorem]{Example}

\numberwithin{equation}{section}
\newcommand{\R}{\mathbb{R}}

\newcommand{\N}{\mathbb{N}}

\newcommand{\I}{\mathbb{I}}

\begin{document}

\title{A new smoothing method for nonlinear complementarity problems involving $ \mathcal{P}_{0}$-function}

\author{El hassene Osmani}\sameaddress{,2, *}\address{Laboratory of Fundamental and Numerical Mathematics, University Ferhat Abbas of Setif 1, Setif, Algeria.}
\author{ Mounir Haddou }\address{INSA Rennes, CNRS, IRMAR-UMR 6625, University of Rennes, Rennes, France.}
\author{Lina Abdallah}\address{Lebanese University, Tripoli, Lebanon.}
\date{\emph{Keywords.} Nonlinear complementarity problems, Newton's method, smoothing functions, $  \mathcal{P}_{0}$-matrix, interior-point methods.}
\author{Naceurdine Bensalem}\sameaddress{1}

\begin{abstract} 
In this paper, we propose a smoothing method to solve nonlinear complementarity problems involving $ \mathcal{P}_{0}$-functions.
We propose a nonparametric algorithm to solve the nonlinear corresponding system of equations and prove some global and local convergence results.\\
We also present several numerical experiments and applications that show the efficiency of our approach.\\
Our main contribution relies in the fact that the regularization parameter $ r $  is considered as a variable and we do not need any complicated strategy to update it.
\end{abstract}
\maketitle

\renewcommand{\thefootnote}{}
\footnotetext{ $^*$Corresponding author: el-hassene.osmani@insa-rennes.fr }

\section{Introduction}
The non-linear complementarity problem (NCP) consists in finding $x
\in \mathbb{R}^n$ satisfying
\begin{equation}\label{eq1}
x \geq 0 \quad F(x) \geq 0 \quad x^TF(x)=0
\end{equation}
where $F:\mathbb{R}^n \rightarrow \mathbb{R}^n$. When $F$ is linear,
problem (\ref{eq1}) reduces to a linear complementarity problem
(LCP).\\
Nonlinear complementarity problems arise in many practical
applications, for example, KKT systems of mathematical
programming problem, economic equilibria, and engineering
design problems, can be formulated as  NCP problems \cite{1, 3}.\\
Different concepts have been developed to study and solve this problems:
reformulation as a system of nonlinear equations or a minimization problem (see \cite{4, 10}). Recently, there have been strong interests in equation reformulation methods for solving the nonlinear complementarity problems. One of the most effective methods is to transform the NCP into semi-smooth equation (NCP functions) and solve using semi-smooth Newton methods. The most well-known NCP functions are the Fisher-Burmeister function \cite{11} and the min function \cite{12}. Another well-known class of algorithms correspends the smoothing methods. The main idea of smoothing approaches is to approximate or regularize the NCP to obtain smooth equations depending on some parameter \cite{13, 17}.\\
In this paper, we present, smoothing approximation scheme to solve
(1). We replace
\begin{equation}\label{202222}
0 \leq x \perp y \geq 0
\end{equation}
by a sequence of smoothed systems of the form 
\begin{equation} \label{773221111}
G_{r}(x,y):=r \psi^{-1} \left[ \psi\left(\frac{x}{r}\right)+\psi\left(\frac{y}{r}\right)\right]=0.
\end{equation}
All the functions and parameters involved in \eqref{773221111} will be explicit later. 
The novelty of our approach is that, we do not need any complicated strategy to update the regularization parameter $r$ since we will consider it as a new variable. To solve the smoothed equations system we will use standard Newton-like method. Without requiring strict complementarity assumption at the solution of equation \eqref{eq1}, we prove that  the proposed algorithm is well defined, globally and superlinearly convergent. At the end of the paper, we present  numerical results to prove the effectiveness of the algorithm. \\
This paper is organized as follows: some definitions are
introduced in section 2. We present our approximation and
formulation in section 3. In section 4, we discuss our approach and scheme to solve \eqref{eq1}. The convergence properties of the algorithm are given in section 5. The section 6 is devoted to the numerical results with a comparison of our method with other approaches.

 \section{Preliminaries and Problem Setting }
 Consider the nonlinear complementarity problem NCP, which is to find a solution of the system:
 \begin{equation} \label{1}
 x \geq 0,\quad F(x) \geq 0 \quad \text{and} \quad x^{T} F(x)=0 \quad \text{or} \quad  0\leq x ~\bot~ F(x)\geq 0
 \end{equation}
 where $  F$ : $ \R^{n}\longrightarrow \R^{n} $ is a continuous function satisfying some additional assumptions to be precised later.\\
 From \eqref{1}, we obtain the equivalent formulation for componentwise products 
 \begin{equation*}
 x\geq 0, ~F(x) \geq 0~~~~ ~ ~~\quad~ ~ \text{and}~~ ~ \quad ~~x_{i}F_{i}(x)=0, \quad \forall i=1,2,...n.
 \end{equation*}
 Or equivalently 
 \begin{equation*}
 x.F(x)=0, \quad~~ x\geq 0, ~F(x)\geq 0,
 \end{equation*}
 where $ "." $ stands for the Hadamard product. It provides an explanation for the term "complementarity", namely, for all $ i=1,2,...,n,~x_{i}$ and $ F_{i}(x) $ are complementary in the sense that if one of them is positive then the other term must be zero.\\
A particular and important class  of NCP is the LCP class defined below.
 	\begin{definition} (Linear complementarity problems)\\
 		When $ F $ is affine function:
 		\begin{equation*}
 		F(x)=Mx+q, \quad x\in \R^{n},~ q\in \R^{n},~ M \in \R^{n\times n}.
 		\end{equation*}
 		The corresponding  NCP is called a linear complementarity problem(denoted LCP). So an problem is to find  $ x \in \R^{n} $ such that 
 		\begin{equation*}
 		x \geq 0, \quad Mx+q \geq 0\quad \text{and}\quad x^{T}(Mx+q)=0.
 		\end{equation*}
 	\end{definition}
 To solve NCP, there are essentialy three different classes of methods: equation-based methods (smoothing), merit functions and projection-tupe methods.
 Our goal in this paper is to present new and very simple smoothing and approximation schemes to solve NCP and to produce efficient numerical methods. In our approach we do not need any complicated strategy to update the smoothing  parameter since we will consider it as a new variable.\\
 First, let us introduce usual assumptions on $ F $ and the ones that will be used in this paper. A well known and studied situation corresponds to monotone functions $ F $ and several methods and algorithms have been developed in this case.\\
 Almost all the solution methods consider at least the following important and standard condition on the mapping $ F $ (monotonicity): We recall that $ F $ is said to be monotone if $ F : \R^{n} \to \R^{n}$ satisfies for any $ (x, y) \in \R^{n},$ 
 \begin{equation*}
 	(x-y)^{T}(F(x)-F(y))\geq 0.
 \end{equation*} 
 In this work, we will consider a weaker assumtion on $ F$:
 \begin{equation*}
 (H_{0})  \quad \quad \quad F ~~~\text{is a} ~~~ \mathcal{P}_{0}\text{-function}
 \end{equation*}
 to prove the convergence of our approach. We recall the following definitions of$ \mathcal{P}_{0} $ and $ \mathcal{P} $-function.  We say that  $ F : \R^{n} \to \R^{n} $  is  a $ \mathcal{P}_{0}$-function (respectively  $ \mathcal{P} $ functions ) if $ \forall x, y \in \R^{n} $ with $  x \ne y,$ there exists an index $ i_{0} \in \{1, 2, ..., n\}$ such that 
 \begin{equation*}
  (x_{i_{0}}-y_{i_{0}})[F_{i_{0}}(x)-F_{i_{0}}(y)] \geq0,
 \end{equation*}

 \begin{equation*}
(\text{respectively } \quad (x_{i_{0}}-y_{i_{0}})[F_{i_{0}}(x)-F_{i_{0}}(y)] >0).
 \end{equation*}
 It is important to notice that the index $ i_{0}$ can depend on $ x $ and $ y $.
 \section{ smoothings approximation functions}
 In this section, we present a new smoothing function for NCP. A function $ \phi : \R^{2} \to \R $ is said to be a NCP function if $ \phi $ satisfies 
 \begin{equation}
  \phi(a, b)=0  \iff a\geq 0,~ b\geq 0,~ ab=0.
 \end{equation}
 An example of such function is $ \phi_{\min}: \R^{2} \to \R$,
 \begin{equation*}
 \phi_{\min}(a, b)= \min \{a, b\}.
 \end{equation*}
Then, $ \phi_{\min} $ is a NCP function. Problem \eqref{1} is then  equivalent to the following system of nonlinear equations:
 \begin{equation}\label{2}
 H(x)=\left(
 \begin{array}{llllll} 
 \phi_{min}(x_{1}, F_{1}(x))\\
\phi_{min}(x_{2}, F_{2}(x))\\
~~\quad \quad  \vdots \\ 
\phi_{min}(x_{n}, F_{n}(x))
 \end{array}
 \right)=0.
 \end{equation}
 This system is clearly non-smooth; classical Newton-like methods can not be used to try to solve it. To overcome this difficulty, there exist several semi-smooth approaches. These techniques may present difficulties to converge. An efficient approach is to approximate \eqref{2} by a smooth one. The following subsection introduces some smoothing functions and establishes different properties that will be useful for our study. 
 \subsection{$\theta$-smoothing}
 In this section, we elaborate on how such a regularized function can be actually built up from the function \eqref{2}. Our smoothing technique is based on the continuous approximation of a more elementary object, namely the step function. The step function is understood here to be the function $ \Im: \R_{+} \to \{0,~1\}$ defined as 
 
  \begin{equation}
\Im(t)=\left\{
 \begin{array}{llllll}
 0 \quad & \text{if} \quad t= 0,\\ 
 1  \quad & \text{if}  \quad t >0.
 \end{array}
 \right.
 \end{equation}
 As an indicator of positive arguments $ t>0 $ over $ \R_{+} $, the step function $ \Im $ "descriminates" the argument $ t=0 $ by assigning a zero value to it. The price to be paid for this sharp detection is the discontinuity of $ \Im $ at $ t=0.$ We wish to have a regularization of $ \Im, $ that is, a family of functions 
 \begin{equation} \label{98765121212}
 	\{\tilde{\Im}(.,r): \R_{+} \to [0, 1), r>0\},
 \end{equation} 
 such that
 \begin{itemize}
 	\item $ \tilde{\Im}(.,r) $  is a smooth function of $ t \geq 0, $ for all $ r>0; $
 	\item  $ \tilde{\Im} $ is continuous with respect to $ r, $ in some functional sense;
 	\item lim$ _{r \downarrow 0} ~\tilde{\Im}(.,r) =\Im(.), $ in some functional sense. 
 \end{itemize}
 To obtain such a family, we follow the methodology developed by Haddou and his coauthors \cite{HADDOU, 27}, the key ingredient of which is a smoothing function. This notion turned out to be a versatile tool in a wide variety of pure and applied mathematical problems \cite{M. Haddou and P. Maheux, 28, 29, 30}. We begin with a "father" function, from which all other regularized functions will be generated.
  \begin{definition}
  ($ \theta$-smoothing function). A function $ \theta : \R \to [0, 1) $ is said to be a $ \theta$-smoothing function if it is continuous, nondecreasing, concave, and 
  \begin{equation} \label{12344321}
  	\begin{array}{llllll}
  	~~~~~~~~ \theta(0)=0,\\
  	 \lim\limits_{\substack{ t \rightarrow +\infty}} \theta(t)=1.
  	\end{array}
  \end{equation}
 The two most common examples of smoothing functions are:
 \begin{enumerate}
 	\item the rational function $ \theta^{1}:\R \to (- \infty, 1) $ defined by 
 	\begin{equation} \label{exp1}
 		\theta^{1}(t)=\dfrac{t}{t+1} \quad \text{for} \quad t\geq 0 \quad \text{and} \quad \theta^{1}(t)=t \quad \text{for} \quad t\leq 0.
 	\end{equation}
 
 	\item the exponential  function $ \theta^{2}:\R  \to (- \infty, 1) $ defined by 
 	\begin{equation}\label{exp2}
 	\theta^{2}(t)=1- \exp(-t).
 	\end{equation}
  \end{enumerate}
  \end{definition}
A more general "recipe" to build such function is to consider nonincreasing probability density functions \\
 $f: \R_{+} \to \R_{+} $ and then take the corresponding cumulative distribution function on $ \R_{+} $ i.e.,
\begin{equation}
\theta(t)= \int_{0}^{t}f(y)dy, \quad t \geq 0,
\end{equation}
we complete the definition of $ \theta $ on $ \R_{-} $ by $ \theta(t)=t $ to get a continuous, nondecreasing function. The nonincreasing assumtion on $ f $ gives the concavity of $ \theta.$ Once a favorite $ \theta$-smoothing has been selected, the next step is to dilate or compress it in order to produce a family of regularized functions for the step function $ \Im.$
  \begin{definition}
  	($ \theta$-smoothing family). Let $ \theta$ be a  $ \theta$-smoothing function. The family of functions
  	\begin{equation}
  	\left\{ \theta_{r}(t):=\theta(\frac{t}{r}), ~~r>0\right\},
  	\end{equation}
  	is said to be the $ \theta$-smoothing family associated with $ \theta$.
  \end{definition}
  Obviously, $ \theta_{r} $ is a smooth function of $ t \geq 0 $ for all $ r>0. $ It is also continuous with respect to $ r $ at each fixed $ r \geq 0.$ From the defining properties \eqref{12344321}, it can be readily shown that 
  \begin{equation}
  	 \lim\limits_{\substack{ r \rightarrow 0}} \theta_{r}(t)= \Im(t), \quad \forall t\geq 0.
  \end{equation}
  	In other words, $ \Im $ is the limit of $ \theta_{r} $  in the sense of pointwise convergence. Thus, $ \{\Im(., r)=\theta_{r}, ~r>0\} $ is a good family of regularized functions in the sense of \eqref{98765121212}. Associated with the two examples \eqref{exp1}-\eqref{exp2} are:
  
  \begin{enumerate}
  	\item the rational family $ \theta^{1}_{r}: \R \to (- \infty, 1) $ defined by 
  	\begin{equation} \label{expf1}
  	\theta_{r}^{1}(t)=\dfrac{t}{t+r} \quad  \text{for} \quad  t \geq 0 \quad \text{and} \quad \dfrac{t}{r} \quad \text{for} \quad t\leq 0.
  	\end{equation}
  	
  	\item the exponential  family $ \theta^{2}_{r}: \R \to (- \infty, 1) $ defined by 
  	\begin{equation}\label{expf2}
  	\theta_{r}^{2}(t)=1- \exp(-t/r).
  	\end{equation}
  \end{enumerate}
  
 Figure 1 display the two families \eqref{expf1}-\eqref{expf2} for  a few values of the  parameter $ r $. We can see that the smaller $ r $ is, the steeper is the slope at $ t=0 $ and the closer to $ \Im $ the function is. 
 
 \begin{figure}[H]\label{fig1}
 	\begin{minipage}[H]{.46\linewidth}
 		\begin{center}
 			\includegraphics[width=9cm,height=7cm]{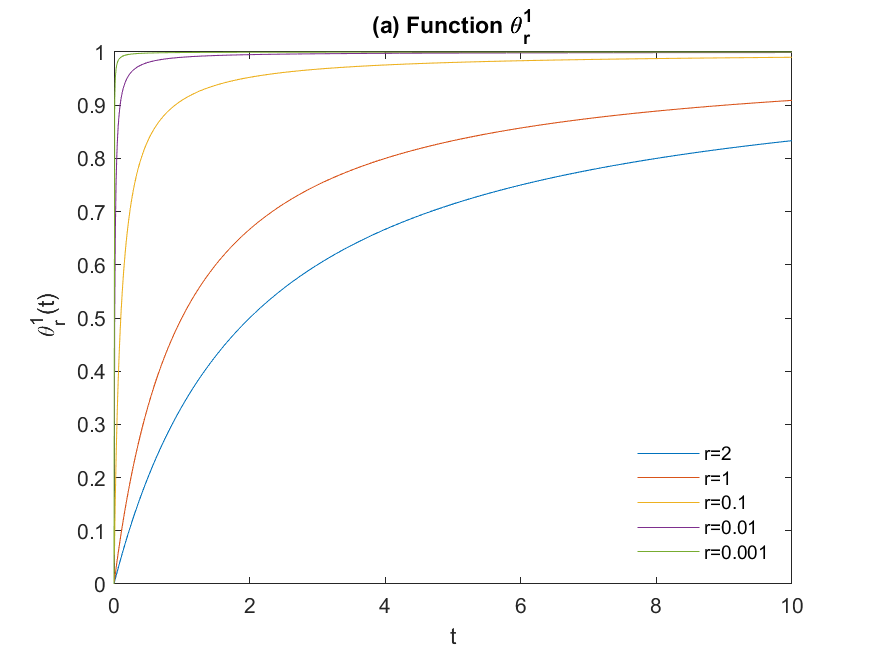}
 		\end{center}
 	\end{minipage} \hfill
 	\begin{minipage}[H]{.46\linewidth}
 		\begin{center}
 			\includegraphics[width=9cm,height=7cm]{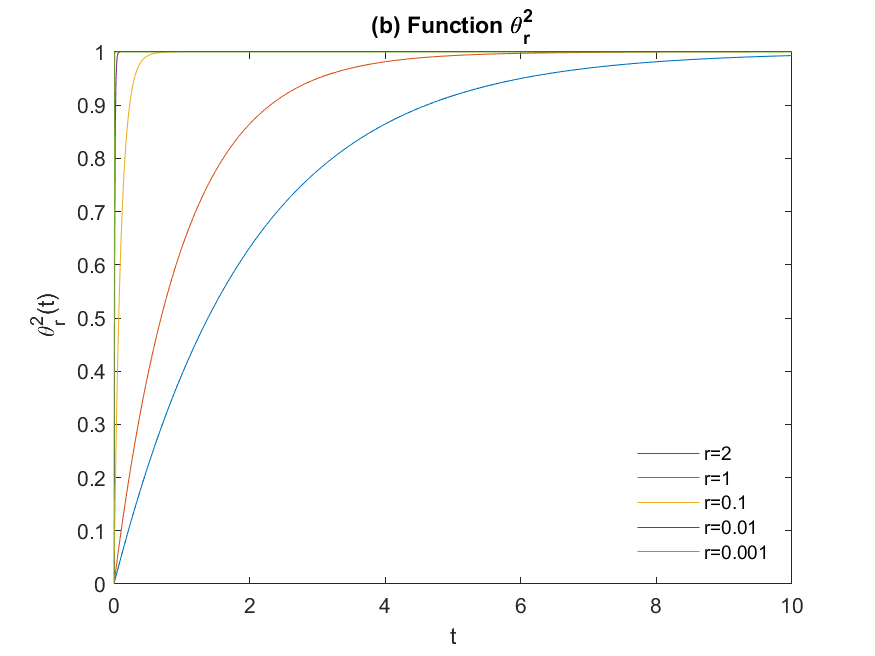}
 		\end{center}
 	\end{minipage}
 	\caption{Function $ \theta_{r}$ for a few values of $ r $.}
 \end{figure}
 \subsubsection{$ \theta $-smoothing of a complementarity condition}
 A $\theta$-smoothing function paves the way for a smooth approximation of a complementarity condition. \\
 Let $ (x,z)\in \R^{2}$ be two scalars such that
 \begin{equation} \label{233323}
 0 \leq x \perp z \geq 0, 
 \end{equation}
 that is,
 \begin{equation*}
 x\geq 0,\quad z\geq 0, \quad xz=0.  
 \end{equation*}
 In the $ (x,z)$-plane, the set of points obeying \eqref{233323} is the union of the two semi-axes  $ \{x \geq 0,~z=0\}~$ and \\$ \{x=0,~ z\geq 0\}.$ Visually, the nonsmoothness of \eqref{233323} is manifested by the "kink" at the corner $ (x,z)=(0,0).$ It is also clear that the corresponding set is non-convex. We consider two possible smooth approximations of \eqref{233323}, depending how it is rewritten in terms of the step function $ \Im $.
 \begin{lemma}
 	\cite{HADDOU} Assuming $ x\geq 0 $ and $ z\geq 0 $, we have the equivalence  
 	\begin{equation}\label{889923}
 	xz=0 \iff \Im(x)+\Im(z) \leq 0
 	\end{equation}
 \end{lemma}
The equivalence \eqref{889923} suggests us to impose 
\begin{equation} \label{777333}
	 x \geq 0, \quad  z  \geq 0, \quad \theta_{r}(x)+\theta_{r}(z) \leq 1,
\end{equation}
for $ r>0,$ as a smooth approximation of \eqref{233323}. Replacing $ \Im $ by $ \theta_{r}$ in \eqref{889923} is logical. Replacing "$ \leq $" by "$ = $" in \eqref{889923} and ther \eqref{777333} seems to be a bold move, but this is motivated by the fact that we want an equality to be mounted into the system of equations. Some times an additional assution (strict complementarity $ x+z>0 $) is made to get such equations.
 \begin{lemma}\label{120967}
 	\cite{HADDOU}
 	Assuming $ x\geq 0 $ and $ z\geq 0 $, we have the equivalence
 	\begin{equation}\label{8333213}
 		xz=0 \iff \Im(x)+\Im(z) =\Im(x+z).
 	\end{equation}	
 \end{lemma}
 The equivalence \eqref{8333213} suggests us to impose 
   
 \begin{equation} \label{29292929210}
 x \geq 0, \quad  z  \geq 0, \quad \theta_{r}(x)+\theta_{r}(z)=\theta_{r}(x+z),
 \end{equation}
 for $ r>0,$ as a smooth approximation of \eqref{233323}.\\
For both approximations \eqref{777333} and \eqref{29292929210}, the paramater $ r $ has to be driven to zero.  
 \subsection{A New smoothing function using $ \theta$-function}
 Our aim is to propose a large class of $ \theta$-functions for which the problems
 \begin{equation} \label{55}
 	 x^{(r)} \geq 0, ~~F(x^{(r)})\geq 0~~ \quad \text{and} \quad ~~\theta_{r}(x^{(r)})+\theta_{r}(F(x^{(r)}))=1,
 \end{equation} 
 are well posed and any limit point of $ (x^{(r)}) $ when $ r $ goes to $ 0,$ is a solution of (NCP).
 In the multidimensional case, the equation just above has to be interpreted as a system of $ n$ equations,
  \begin{equation*}
  \theta_{r}(x_{i}^{(r)})+\theta_{r}(F_{i}(x^{(r)}))=1, \quad   i: 1...n.
 \end{equation*}
 Note that the relation \eqref{55} is symmetric in $ x $ and $ F(x) $. Thus, our problem can be seen as a fixed point problem for the function $ F_{r, \theta}(x) $ defined just below. Indeed, the equation \eqref{55} is equivalent to 
 
 \begin{equation*}
 x=\theta^{-1}_{r}(1-\theta_{r}(F(x)))=r\theta^{-1}(1-\theta(F(x)/r))=: F_{r, \theta}(x).
 \end{equation*}
 By symmetry of the equation \eqref{55}, we also have the relations: 
 \begin{equation*}
 F(x)=\theta_{r}^{-1}(1-\theta_{r}(x))=r\theta^{-1}(1-\theta(x/r)),
 \end{equation*}
 but we shall not go that direction. We propose another way to approximate a solution of the (NCP) problem as follows (see \cite{M. Haddou and P. Maheux}). Let $ \psi_{r}(t)=1-\theta_{r}(t) $, the relation \eqref{55} is equivalent to the three following equalities
 \begin{equation*}
 \begin{array}{llllll}
\psi_{r}(x)+\psi_{r}(F(x))=1=\psi_{r}(0), \\
\\
 \psi_{r}^{-1} \left[\psi_{r}(x)+\psi_{r}(F(x))\right]=0 \quad \text{and}\\
 \\
 r \psi^{-1}\left[\psi\left(\frac{x}{r}\right)+\psi\left(\frac{F(x)}{r}\right)\right]=0.
 \end{array}	
 \end{equation*} 
 (with $ \psi=\psi=1-\theta$). For the sequal, we set for any $ x, y \in \R^{n} $ and any $ r>0 $
 \begin{equation} \label{77}
 G_{r}(x,y):=r \psi^{-1} \left[ \psi\left(\frac{x}{r}\right)+\psi\left(\frac{y}{r}\right)\right].
 \end{equation}
 First, we characterize the solutions $ (x, y) $ of $ G_{r}(x,y) =0$ when $ \psi $ satisfies some conditions independent of $ F $. \\
 Let $ 0<a<1.$ We say that $ \psi $ satisfies $ (H_{a}) $ if there exists $ s_{a}>0 $ such that, for all $ s \geq s_{a}, $
\begin{equation*}
	\psi(s) \leq \frac{1}{2} \psi(as) \quad \text{or equivalently} \quad \frac{1}{2}+\frac{1}{2}\theta(as) \leq \theta(s). \quad \quad  \quad \quad \quad \quad \quad  \quad  (H_{a})
\end{equation*} 
 The condition $ (H_{a}) $ imposes that the decay of $ \psi(s) $ is under some uniform control for large $ s $ or in terms of $ \theta $ that $ \theta(s) $ should grow enough quickly with some uniformity for large $ s $.
 Since $ \psi $ and $ \theta $ are monotone, it is interesting to take $ a$ as  large as possible in the condition $ (H_{a}) $ since $ (H_{a}) \Longrightarrow (H_{b})$ for $ b<a.$ \\
 Note that we can never take $ a=1 $ because $ \theta \leq 1 $ unless  $ \theta $  is constant and equal to one for large $ s $. But in some cases, $ a $ can be chosen as close to $ 1,$ see for instance $ \theta^{2}.$\\
 One can obtain by simple calculations that:
 
 \begin{enumerate}
	\item  For $ \theta^{1},$ we have
 \begin{equation*}
 \psi^{1}(t)=\left\{
 \begin{array}{llllll}
 \dfrac{1}{t+1} \quad & \text{if} \quad t \geq 0,\\ 
  	1-t  \quad & \text{if}  \quad t <0,
 \end{array}
 \right.
 \end{equation*}
 and the condition $ (H_{a}) $ is only satisfied for $ 0<a<1/2 $ with $ s_{a}\geq \dfrac{1}{1-2a}$.\\
 
 	\item For $ \theta^{2},$ we have $ \psi(t)=e^{-t} $ and the condition $ (H_{a})$ is satisfied for any $ 0<a<1 $ with $  s_{a}=\dfrac{\ln 2}{1-a}.$ 
 \end{enumerate}
 From now on, all the results use the function $ \psi $. Obviously, everything can be easily transposed on $ \theta.$ The following Lemma compare the function $ G_{r} $ defined in \eqref{77} to the $ \min$ function and will be useful for the rest of our analysis. 
 \begin{lemma}
 If $ \psi : \R \to ]0, +\infty[ $ is an invertible non-increasing function, then for any $ (s, t) \in \R^{2}$ and any  $ r>0$  \\
 \begin{equation*}
 	G_{r}(s,t) \leq \min(s,t).
 \end{equation*}
 \end{lemma}
 \begin{proof}
 Let $ s, t \in \R $ be fixed. By symmetry, we can assume that $ s=min(s,t).$ Since $ \psi \geq 0,$ we obviously have 
 \begin{equation*}
 \psi(s/r)\leq \psi(s/r)+\psi(t/r).
 \end{equation*}
 By the fact that $ \psi $ is invertible and non-increasing, we get 
 \begin{equation*}
 \psi^{-1}(\psi(s/r)+\psi(t/r))\leq s/r.
 \end{equation*}
 Thus, from the definition of $ G_{r} $ we conclude that 
 \begin{equation*}
 G_{r}(s,t)=r\psi^{-1}[\psi(s/r)+\psi(t/r)]\leq s=min(s,t).
 \end{equation*}
\end{proof}
 The next theorem shows how the condition $ (H_{a})$ gives information about the behavior of $ G_{r}.$
 \begin{theorem}
 Let $ \psi: \R \to ]0, + \infty [$ be an invertible non-increasing function such that 
 
 \begin{equation*}
  \lim\limits_{\substack{ t \rightarrow - \infty}}\psi(t)=+\infty, \psi(0)=1, \text{and}  \lim\limits_{\substack{t \rightarrow + \infty}}\psi(t)=0.
 \end{equation*}
 
 If $ \psi $ satisfies the condition $ (H_{a})$ for some $ a \in~ ]0, 1[,$ then for all $ s,t \in \R, $
 \begin{equation*}
 	\lim\limits_{\substack{r \searrow 0}}G_{r}(s,t)=0 \iff \min(s,t)=0.
 \end{equation*}
 \end{theorem}

\begin{proof}
We start by the direct implication \\
Let $ s, t \in \R $ be fixed. By Lemma 3.5, for any $ r>0 $ we have $ G_{r}(s,t)\leq \min(s,t) $ and, then $ \min(s,t) \geq 0$. We finish the proof by contradiction as follows. Asumme that $ s=\min(s,t) >0.$\\
Since $ \psi $ is nonincreasing and $ s \leq t,$ we have 
\begin{equation*}
\psi(s/r)+\psi(t/r) \leq 2\psi(s/r).
\end{equation*}
By assumption $ (H_{a}) $ and for $ r $ small enough, $ 2\psi(s/r) \leq \psi(as/r)$. Indeed, the ratio $ s/r$ goes to infinity as $ r $ goes to $ 0 $ because $ s>0 $. Hence 
\begin{equation*}
\psi(s/r)+\psi(t/r) \leq \psi(as/r).
\end{equation*}
Now since $ \psi^{-1}$ is nonincreasing, 
\begin{equation*}
as/r \leq \psi^{-1}(\psi(s/r)+\psi(t/r)),
\end{equation*}
or equivalently with $ r $ small enough, $ s\leq a^{-1}G_{r}(s,t).$\\
Passing to the limit, $ \lim\limits_{\substack{r \searrow 0}}G_{r}(s,t)=0 $ and , then $ s \leq 0 $ in contradiction with $ s>0. $ \\
Now, we prove the converse ($ \Leftarrow $):\\
Assume $ s=\min(s,t)$. Hence, $ s=0. $ Since $ \psi(0)=1,$ we have 
\begin{equation*}
G_{r}(s,t)=r\psi^{-1}(1+\psi(t/r)) .
\end{equation*}
If $ t=0 $ then $ \lim\limits_{\substack{r \searrow 0}}G_{r}(s,t) =$	$ 	\lim\limits_{\substack{r \searrow 0}}r\psi^{-1}(2)=0.$ \\
 If $ t>0 $ then $ \lim\limits_{\substack{r \searrow 0}}\psi(t/r)=0.$ Thus $ \lim\limits_{\substack{r \searrow 0}}G_{r}(s,t)=0$ by continuity of $ \psi^{-1}.$\\
In both cases, we have $ \lim\limits_{\substack{r \searrow 0}}G_{r}(s,t)=0 $. 
\end{proof}
For both $ \theta_{r}^{1}  $ and $ \theta_{r}^{2}  $ examples, the assertion of Theorem 3.6 is clearly satisfied. Indeed direct computations lead to 
\begin{enumerate}
	\item For $ s>0 $ and  $ t>0 $ such that $ \dfrac{1}{s}+\dfrac{1}{t}\leq \dfrac{1}{r},$ we have the following explicit expression 
	\begin{equation}
		G_{r}^{1}(s,t)=\dfrac{st-r^{2}}{s+t+2r}
	\end{equation}
	Note that the denominator is not zero when $ s,t $ are non-negative even when $ s=t=0$. In addition, when $ \min(s,t)>0 $ we have $ \lim\limits_{\substack{r \searrow 0}}G_{r}^{1}(s,t) =\dfrac{st}{s+t}<\min(s,t).$ 
	\item For any $ s,t  \in \R $, we have the following explicit expression 
	\begin{equation} \label{78781}
		G^{2}_{r}(s,t)=-r\log(e^{-s/r}+e^{-t/r}).
	\end{equation}
	Assume $ s=\min(s,t)$. Then we have $ s-r\log(2) \leq G_{r}^{2}(s,t) $ because 
	\begin{equation*}
		e^{-s/r}+	e^{-t/r} \leq 2	e^{-s/r}.
	\end{equation*}
	Thus, $ \min(s,t)-r\log(2)\leq G_{r}^{2}(s,t)\leq \min(s,t).$ Passing to the limit as $ r $ goes to 0, we conclude that $ \lim\limits_{\substack{r \searrow 0}}G_{r}^{2}(s,t)=\min(s,t).$ 
\end{enumerate}
Figure 2. illustrate the behaviour of $ G_{r}(x,-x)$.
\begin{figure}[H]\label{fig1}
	\begin{center}
		\includegraphics[width=9cm,height=7cm]{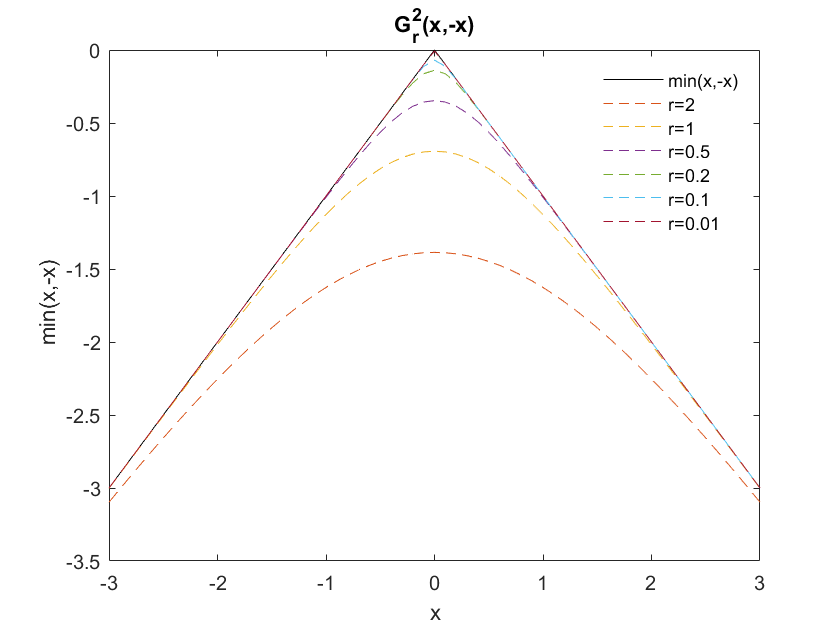}
	\end{center}
	\caption{Comparaison of $ G_{r}(x,-x) $ and $ \min(x,-x) $.}
\end{figure}

Now, we focus on the case where $ \psi $ satisfies $ (H_{a}) $ for all $ a\in ]a, 1[ $ and prove a stronger result.
\begin{theorem}
Let $ \psi: \R \to ]0, + \infty [$ be an invertible non-increasing function such that 

\begin{equation*}
\lim\limits_{\substack{ t \rightarrow - \infty}}\psi(t)=+\infty, \psi(0)=1, \text{and}  \lim\limits_{\substack{t \rightarrow + \infty}}\psi(t)=0.
\end{equation*}

If $ \psi $ satisfies  $ (H_{a})$ for all $ a \in~ ]0, 1[,$ then for any $ s,t >0, $
\begin{equation*}
\lim\limits_{\substack{r \searrow 0}}G_{r}(s,t)=\min(s,t).
\end{equation*}
\end{theorem}

\begin{proof}
	By Lemma 3.3, we have 
	\begin{equation*}
		\forall r>0, \quad \forall s,t \in \R, \quad G_{r}(s,t) \leq \min(s,t)
	\end{equation*}
	Thus, we have to concentrate on the lower bound of $ G_{r}$.\\
	Let $ s,t>0 $ such that $ s=min(s,t)>0.$ For each $ a \in ]0, 1[$ and when $ r $ is sufficiently small (i.e. $ s/r \geq s_{a}>0 $) we can apply the assumption $ (H_{a}) $ to get 
\begin{equation*}
\psi(s/r)+\psi(t/r) \leq 2\psi(s/r) \leq \psi(as/r).
\end{equation*}
Since $ \psi^{-1} $ is nonincreasing, we deduce 
\begin{equation*}
as/r \leq \psi^{-1}(\psi(s/r)+\psi(t/r)).
\end{equation*}
Thus, for any $ a \in ]0, 1[$ and any $ 0 < r < s/s_{a}, $ we have $ as \leq G_{r}(s,t),$ \\
Hence,
\begin{equation*}
a\min(s,t)=as \leq \lim\limits_{\substack{r \searrow 0}} \inf G_{r}(s,t) \leq \lim\limits_{\substack{r \searrow 0}} \sup G_{r}(s,t) \leq \min(s,t).
\end{equation*}

By taking $ a \nearrow 1, $ we obtain the desired result.
\end{proof}

\subsection{An approximate formulation}
In this section, we present two new  reformulations of the complementarity problem \eqref{1} by using two approximations, the first with the $ \theta^{1}_{r}$-function and the second with the  $ \theta^{2}_{r}$-function.
\subsubsection{ Approximation of NCP using $ \theta^{1}_{r}$-function}
Using $ \theta_{r}^{1} $ function, we regularize each complementarity constraint by considering 
\begin{equation*}
x_{i}z_{i}=0, ~~~ \text{by} \quad~~~ 	G_{r}^{1}(x_{i},z_{i})=\dfrac{x_{i}z_{i}-r^{2}}{x_{i}+z_{i}+2r}=0,~~~~~~~\quad \forall~  i=1,...n.
\end{equation*}
This approximation yields to we obtain the following formulation:
\begin{equation}
(\tilde{P}_{\theta_{1}})~~~~~~\left\{
\begin{array}{llllll} 
F(x)=z,  \\
x \geq 0,~~~ z \geq 0,~~~ r\searrow 0, \\
G_{r}^{1}(x,z)=0.
\end{array}
\right.
\end{equation}
Where $ G_{r}^{1}(x,z)=\dfrac{xz-r^{2}}{x+z+2r}.$ Here, it is understood that $ G_{r}^{1}$ operates componentwise on $ x$ and $z.$ We consider the family $ \{H_{\theta_{1}}^{r}(.),~r>0\}$,  where 
\begin{equation}\label{10}
H_{\theta_{1}}^{r}(x,z)=\left[
\begin{array}{llllll} 
F(x)-z\\
G_{r}^{1}(x,z)
\end{array}
\right],
\end{equation}
is a regularized function of $ H $ defined in \eqref{2}. 

\begin{lemma}
	Let $ H_{\theta_{1}}^{r}(x,z) $ define by \eqref{10}. Then, the Jacobian matrix of  $ H_{\theta_{1}}^{r}(x,z)$ is 
	\begin{equation*}
	\nabla H_{\theta_{1}}^{r}(x,z) =\left(\begin{array}{cc}
	\nabla F(x)&-I\\
	D_{a}(x,z)& D_{b}(x,z)
	\end{array}\right)
	\end{equation*}
	where  $ D_{a}(x, z)= \text{diag}\{a_{1}(x, z), ...,a_{n}(x, z) \}$ and $ D_{b}(x, z)= \text{diag}\{b_{1}(x, z), ...,b_{n}(x, z) \}$ are two diagonal matrices, given by \\
	\begin{equation*}
	a_{i}(x, z)=\left(\dfrac{z_{i}+r}{x_{i}+z_{i}+2r}\right)^{2}, \quad \quad  \quad  \quad  b_{i}(x, z)=\left(\dfrac{x_{i}+r}{x_{i}+z_{i}+2r}\right)^{2}, \quad  \quad  \quad  i=1, ...n.
	\end{equation*}

\end{lemma}

\subsubsection{ Approximation of NCP using $ \theta^{2}_{r}$-function}
Using the $ \theta^{2}_{r}$-function defined above, we obtain an approximate formulation for NCP 
\begin{equation}
(\tilde{P}_{\theta_{2}})~~~~~~\left\{
\begin{array}{llllll}  
F(x)=z,  \\
x \geq 0,~~~ z \geq 0,~~~ r\searrow 0 \\
G_{r}^{2}(x,z)=0.
\end{array}
\right.
\end{equation}
Where $ G_{r}^{2}(x,z)=-r\log(e^{-x/r}+e^{-z/r}) $, by the same way as for \eqref{10}, $ G_{r}^{2} $  operates componentwise on $ x $ and $ z $. We consider the family $ \{H_{\theta_{2}}^{r}(.),~r>0\},$ where 
\begin{equation}\label{11}
H_{\theta_{2}}^{r}(x,z)=\left[
\begin{array}{llllll} 
F(x)-z\\
G_{r}^{2}(x,z)
\end{array}
\right],
\end{equation}	
is a regularized function of $ H$ defined in \eqref{2}.

\begin{lemma}
	Let  $ H_{\theta_{2}}^{r}(x,z) $ define by \eqref{11}. Then, the Jacobian matrix of   $  H_{\theta_{2}}^{r}(x,z)$ is	
\begin{equation*}
\nabla H_{\theta_{2}}^{r}(x,z)=\left(\begin{array}{cc}
 \nabla F(x)&-I\\
Q_{k}(x, z)& Q_{l}(x, z)
\end{array}\right),
\end{equation*}
where $ Q_{k}(x, z)= \text{diag}\{k_{1}(x, z), ...,k_{n}(x, z) \}$ and $ Q_{l}(x, z)= \text{diag}\{l_{1}(x, z), ...,l_{n}(x, z) \}$  are two diagonal matrices, given by
\begin{equation*}
k_{i}(x,z)=\dfrac{e^{-x_{i}/r}}{e^{-x_{i}/r}+e^{-z_{i}/r}},\quad \quad  \quad  \quad  l_{i}(x,z)=\dfrac{e^{-z_{i}/r}}{e^{-x_{i}/r}+e^{-z_{i}/r}}, \quad  \quad  \quad  i=1, ...n.
\end{equation*}
\end{lemma}
In order to prove the nonsingular of the Jacobian matrix of $ H_{\theta_{1}}^{r}(x,z) $ (resp. $ H_{\theta_{2}}^{r}(x,z) $) , we state
first a basic but essential lemma.
\begin{lemma}
	Let $ M \in \R^{n \times n}$ be a $ \mathcal{P}_{0}-$matrix. Then any matrix in the following form is nonsingular:
	\begin{equation*}
	N_{s}+N_{t}M,
	\end{equation*}
	where $ N_{s}\in \R^{n \times n}$ is a positive (negative) diagonal matrix, and $ N_{t}\in \R^{n \times n}$ is a nonnegative (non-positive) diagonal matrix.
\end{lemma}
\begin{proof}
	Let $ N_{s}=\text{diag}(s_{1}, s_{2}, ..., s_{n}) $ and $ N_{t}=\text{diag}(t_{1}, t_{2}, ..., t_{n})$. If $ N_{s}$ is positive, and $ N_{t}$ is nonnegative, then $ s_{i}>0 $ and $ t_{i}\geq 0 $ for all $ i=1, 2, ..., n.$\\
	Let $ v \in \R^{n} $ be a vector such that $ (N_{s}+N_{t}M)v=0.$ Then, we have $ v_{i}=-\dfrac{t_{i}}{s_{i}}(Mv)_{i}.$ \\
	It yields $ v_{i}^{2}=-\dfrac{t_{i}}{s_{i}}v_{i}(Mv)_{i}.$ If $ t_{i}=0,$ then $ v_{i}=0,~~ \forall i=1, ..., n.$ \\
	If $ v_{i} \ne 0,$ we have $\frac{t_{i}}{s_{i}} >0.$ Owing to $ v_{i}^{2} \geq 0,$ we have $ v_{i}(Mv)_{i} \leq 0.$ If $ v_{i}(Mv)_{i} = 0,$ then $ v_{i}=0.$ Otherwise, $ v_{i}(Mv)_{i} < 0 $ contradicts the  property of $ M $. Based on the above discussion, it is concluded that $ v=0,$ then $N_{s}+N_{t}M $ is a nonsingular matrix.
\end{proof}
By Lemma 3.10, we can obtain a property of ${H}_{\theta_{1}}^{r}$ and  and ${H}_{\theta_{2}}^{r}$  if $ F $ is a $\mathcal{P}_{0}$-matrix.
\begin{theorem} 
	Let $ F$ be a $\mathcal{P}_{0}$-function. Then, for any $ r>0,$ and $ (x, z)\in \R^{2n} $  the Jacobian matrix $ \nabla {H}_{\theta_{1}}^{r}(x,z)  $ (resp. $\nabla {H}_{\theta_{2}}^{r}(x,z)$)  is nonsingular. 
\end{theorem}
\begin{proof}
	For all $ r>0$ and from Lemma 3.8 and Lemma 3.9, it follows that the diagonal matrix $ D_{a}(x,z)$ (resp. $ Q_{k}(x,z)$) is non-negative, and  $ D_{b}(x,z)$ (resp. $ Q_{l}(x,z)$)  is non-negative diagonal matrix. \\
	Since $ F $ is a $\mathcal{P}_{0}$ function, the Jacobian matrix $ \nabla F(x) $ is a $\mathcal{P}_{0}$ matrix. \\
	We have 
	\begin{equation*}
	\text{det}(\nabla{H}_{\theta_{1}}^{r}(x,z))=\text{det}( D_{a}(x,z)+\nabla F(x)D_{b}(x,z)),
	\end{equation*}
	and 
	\begin{equation*}
\text{det}( \nabla {H}_{\theta_{2}}^{r}(x,z))=\text{det}( Q_{k}(x,z)+\nabla F(x)Q_{l}(x,z)).
	\end{equation*}
	Since $ \nabla F(x) $ is a $\mathcal{P}_{0}$-matrix and from Lemma 3.10, it follows that $ D_{a}(x,z)+ \nabla F(x)D_{b}(x,z)$ (resp. $ Q_{k}(x,z)+\nabla F(x)Q_{l}(x,z)$) is nonsingular. Hence $ \nabla {H}_{\theta_{1}}^{r}(x,z) $ (resp. $  \nabla {H}_{\theta_{2}}^{r} $) is nonsingular. 
\end{proof}

\section{New approach for solving nonlinear complementarity problems}
In this section, we present the idea of our algorithms for optimization problems to solve the NCP, but here we don't have any objective function to minimize. Our methods take inspiration from Interior Point Methods. We restrict our choice of $ \theta$-fucntion to $ \theta_{r}(x)=\theta^{1}_{r}(x)=\frac{x}{x+r}$. \\
We recall that the interior-point methods have replaced the original nonsmooth problem NCP by a sequence of regularized problems
\begin{equation} \label{12345543211234}
H_{r}(\mathbf{X})=0,
\end{equation}
where 
\begin{equation}
\mathbf{X}=\left [
\begin{array}{llllll} 
x \\
z
\end{array}\right ] \in \R^{2n}_{+}, \quad H_{r}(\mathbf{X})=\left [
\begin{array}{llllll} 
F(x)-z \\
x.z-r\mathbf{1}
\end{array}\right ],
\end{equation}
where $ r \geq 0 $ is the smoothing parameter, $\mathbf{1}  \in \R^{n}$ is the vector whose components are all equal to 1.
The Jacobian matrix of $ H_{r} $ with respect to $\mathbf{X}$, does not depend on $ r $ and can be denoted by   
\begin{equation} \label{0854356}
\nabla H_{r}(\mathbf{X})=\left(\begin{array}{cc}
\nabla F(x)&-I\\
Z& X
\end{array}\right),
\end{equation}
where $ Z=\text{diag}(z) $ and  $ X=\text{diag}(x)$, i.e. the diagonal matrix of $ z $ (resp. $ x $).
\subsection{When the parameter becomes a variable}
In the system \eqref{12345543211234}, the status of the parameter $ r $ is very distinct from that of the variable $ \mathbf{X}$. While $ \mathbf{X}$ is computed "automatically" by a Newton iteration, $ r$ has to be updated "manually" in an ad-hoc manner.  \\
Our goal is to find a strategy that decreases $ r $ during iterations and ensures the nonnegative of variables. However, we must adjust the strategy when the model or its parameters are changed. To avoid this trouble, we consider $ r $ as an unknown of the system instead of a parameter. \\ We feel that it would be judicious to incorporate the parameter $ r$ into the variables.
Let us therefore consider the enlarged vector of unknowns 
\begin{equation}
\mathbb{X}= \left [\begin{array}{llllll} 
\mathbf{X} \\
r
\end{array}\right] \in \R^{2n} \times \R_{+},
\end{equation}
and then consider a system  of $ 2n+1 $ equations
\begin{equation}
\mathbb{H}_{\theta}(\mathbb{X})=0,
\end{equation}
to be on $\mathbb{X}$. To this end, let us remind ourselves that our ultimate goal is to solve $ H_{\theta_{1}}^{0}(\mathbf{X})$ , together with the inequalities $ x\geq 0,~z\geq 0. $ \\
Thus, it is really natural to first consider 
\begin{equation} \label{8893}
\mathbb{H}_{\theta}(\mathbb{X})=\left[\begin{array}{llllll} 
H_{\theta_{1}}^{r}(\mathbf{X})\\
\quad r
\end{array}
\right].
\end{equation}
where
\begin{equation*}
H_{\theta_{1}}^{r}(\mathbf{X})=\left[
\begin{array}{llllll} 
F(x)-z\\
G_{r}^{1}(\mathbf{X})
\end{array}
\right], \quad   \text{and} ~ \quad G_{r}^{1}(\mathbf{X})=\dfrac{xz-r^{2}}{x+z+2r}.
\end{equation*}
This construction turns out to be to naive. Indeed, if we start from some $ r^{0} $ and solve the smooth system \eqref{8893} by the smooth Newton method, since the last equation is linear, we end up with $ r^{1}=0$ at the first iteration. Once the boundary of the interior region is reached, we are "stuck" there.\\
To prevent $ r $ from rushing to zero in just one iteration, we could set   
\begin{equation} \label{8890}
\mathbb{H}_{\theta}(\mathbb{X})=\left[\begin{array}{llllll} 
H_{\theta_{1}}^{r}(\mathbf{X})\\
\quad r^{2}
\end{array}
\right].
\end{equation}
At this stage, system \eqref{8890} is not yet fully adequate. Indeed, the last equation is totally decoupled from the others. Everything happens as if $ r $ follows a prefixed sequence, generated by the Newton iterates of the scalar equation $ r^{2}=0$, regardless of $ \mathbf{X} $. It is desirable to couple $ r $ and $\mathbf{X}$ in a tighter way. In this respect, we advocate
\begin{equation} \label{86689}
\mathbb{H}_{\theta}(\mathbb{X})=\left[\begin{array}{llllll} 
H_{\theta_{1}}^{r}(\mathbf{X})\\
\frac{1}{2} \Vert x^{-}\Vert^{2}+\frac{1}{2} \Vert z^{-}\Vert^{2}+ r^{2}
\end{array}
\right],
\end{equation}
where
\begin{equation*}
\Vert x^{-}\Vert^{2}=\sum_{i=1}^{n} \text{min}^{2}(x_{i},0), \quad \Vert z^{-}\Vert^{2}=\sum_{i=1}^{n} \text{min}^{2}(z_{i},0).
\end{equation*}
This choice has the benefit of taking into account the nonnegativity condition $ x \geq 0 $ and $ z \geq 0 $. \\ Indeed,
the last equation of \eqref{86689} implies that, as long as $ r\geq 0 $, we are ascertained that $ x^{-}=z^{-}=0 $.
This amounts to saying that $ x \geq 0 $ and $ z \geq 0 $. Should a component of $ x $ or $ z $ become negative during the iteration, this equation would contribute to “penalize” it.\\
Since $ r $ is now considered as a variable and  the scalar function $ t \mapsto \frac{1}{2}  \vert \min(t,0)\vert^{2} $ is differentiable and its derivative is equal to $ \min(t,0) $. From this observation, the Jacobian matrix of $\mathbb{H}_{\theta} $ is:   
\begin{equation}
\nabla_{\mathbb{X}}\mathbb{H}_{\theta}(\mathbb{X})=
\begin{pmatrix}
\nabla_{x}H_{\theta_{1}}^{r} & \nabla_{z}H_{\theta_{1}}^{r} & \partial_{r}H_{\theta_{1}}^{r} \\
\\
(x^{-})^{\text{T}}&(z^{-})^{\text{T}}& 2r  \\
\end{pmatrix},
\end{equation}
where $  x^{-}  $ is the vector of components $ x_{i}^{-}=\min(x_{i},0) $ and similarly for $ z^{-},$
\begin{equation*}
\nabla_{x}H_{\theta_{1}}^{r}=\left[\begin{array}{llllll} 
\nabla _{x} F(x)\\
D_{a}(x,z)
\end{array}
\right]_{2n \times n}, \quad \nabla_{z}H_{\theta_{1}}^{r}=\left[\begin{array}{llllll} 
-I\\
D_{b}(x,z)
\end{array}
\right]_{2n \times n}, \quad   \partial_{r}H_{\theta_{1}}^{r}=\left[\begin{array}{llllll} 
\quad \quad  \quad   0_{n \times 1}\\
\text{diag} \left(   \dfrac{-2r}{x+z+2r}+\dfrac{2(r^{2}-xz)}{(x+z+2r)^{2}}\right)  \mathbf{e}
\end{array}
\right]_{2n\times 1}
\end{equation*}
and $ \mathbf{e}  $ is a n-dimensional vector whose entries are equal to $ 1 $.\\
If $ \mathbb{H}_{\theta}(\mathbb{X})=0$ where $  \mathbb{X} \in \R^{2n}_{+} \times \R_{+}$ we obtain $ r=0$ and $ x^{-}=z^{-}=0$. Hence in this case, $ \nabla_{\mathbb{X}}\mathbb{H}_{\theta}(\mathbb{X})$ becomes singular, since $ \text{det}( \nabla_{\mathbb{X}}\mathbb{H}_{\theta}(\mathbb{X}))=0 $. To solve this issue, we add a small enough positive parameter $ \varepsilon $ in the last equation. We get 
\begin{equation} \label{123456}
\frac{1}{2} \Vert x^{-}\Vert^{2}+ \frac{1}{2} \Vert z^{-}\Vert^{2}+r^{2}+\varepsilon r=0.
\end{equation} 
Hence, we define the following systems
\begin{equation} \label{889}
\mathbb{H}_{\theta}(\mathbb{X})=\left[\begin{array}{llllll} 
H_{\theta_{1}}^{r}(\mathbf{X})\\
\frac{1}{2} \Vert x^{-}\Vert^{2}+\frac{1}{2} \Vert z^{-}\Vert^{2}+ r^{2}+\varepsilon r
\end{array}
\right]
\end{equation}
\begin{lemma}
Let $ \mathbf{X} \in \bar{\Xi}$, where $ \Xi $ is the interior region defined in 

\begin{equation}
	\Xi=\{\mathbf{X}=(x,z)\in \R^{2n}\quad |\quad x>0, ~~z>0\}.
\end{equation}
Let $ r\in \R $ and $ \mathbb{X}=[\mathbf{X};r]^{T}.$ Then,
\begin{equation*}
\text{det}~ \nabla \mathbb{H}_{\theta}(\mathbb{X})=(\varepsilon+2r)~\text{det}~ \nabla H_{\theta_{1}}^{r}(\mathbf{X}).
\end{equation*}
If $ r>-\frac{\varepsilon}{2},$ the two Jacobian matrices are singular or nonsigular at the same time.
\end{lemma}
\begin{proof}
Thanks to the assumption $ \mathbf{X} \in \bar{\Xi}, $ we have $ x\geq 0 $ and $ z\geq 0, $ so that $ x^{-}=z^{-}=0. $ Expanding the determinant of \eqref{889} with respect to the last row yields the desired result.
\end{proof}
\section{Convergence}
In this section, we propose a generic algorithm to solve NCP and prove some convergence results.\\
From now on, the enlarged equation $\eqref{889}$ is selected as the reference system in the design of our new algorithm. The idea is simply to apply the standard Newton method to the smooth system $\eqref{889}$. To enforce a global convergence behavior, we also recommend using $ \alpha$ line search like Armijo back-tracking technique.\\ 
Now, we present our algorithm for our method described above:\vspace{0.5cm}\\
\begin{tabular}{llll}
	\rule{1\linewidth}{0.6mm}  \\
	$ \mathbf{Algorithm~1}$ Nonparametric method with Armijo line search \\
	\rule{1\linewidth}{0.6mm}
	\\
	1.~ Chose  $\mathbb{X}^{0} =(\mathbf{X}^{0},r^{0}), ~\mathbf{X}^{0} \in \Xi,~r^{0}=<x^{0},z^{0}>/n,~\tau \in (0,1/2), ~\varrho \in(0,1).~ $ Set  $k=0. $ \\
	2.~ If $ \mathbb{H}_{\theta}(\mathbb{X}^{k})=0,~$stop.\\
	3.~ Find a direction $ \mathbf{d}^{k} \in \R^{2n+1}$ such that\vspace{0.3cm} \\
	\vspace{0.3cm}
	$~~~~~~~~~~~~~~~~\quad\quad\quad\quad\quad\quad\quad\quad\quad\quad\quad\quad\quad\quad\quad~~\quad~~~~~~~~~~\mathbb{H}_{\theta}(\mathbb{X}^{k})+\nabla_{\mathbb{X}}\mathbb{H}_{\theta}(\mathbb{X}^{k})\mathbf{d}^{k}=0.  $\\
	4.~ Choose $ \zeta^{k}=\varrho^{j_{k}}\in (0,1),~ $ where  $ j_{k} \in \N $ is the smallest integer such that \vspace{0.3cm}\\
	\vspace{0.3cm}
	
	$~~~~~~~~~~~~~~~~\quad\quad\quad\quad\quad\quad\quad\quad\quad\quad\quad\quad\quad\quad~~\quad\quad~~~~~~~\Theta(\mathbb{X}^{k}+\varrho^{j_{k}}\mathbf{d}^{k})-\Theta(\mathbb{X}^{k})\leq \tau\varrho^{j_{k}}~ \nabla \Theta(\mathbb{X}^{k})^{T}\mathbf{d}^{k}.  $\\
	5. ~Set $ \mathbb{X}^{k+1}=\mathbb{X}^{k}+\zeta^{k}\mathbf{d}^{k}~ $ and $ k\gets k+1.~ $Go to step $ 2.$ \\
	\\
	\rule{1\linewidth}{0.6mm}
	
\end{tabular}

Where the merit function used in the line search is:
\begin{equation*}
\Theta(\mathbb{X})=\frac{1}{2}\Vert \mathbb{H}_{\theta}(\mathbb{X}) \Vert^{2}.
\end{equation*}

A detailed description of Nonparametric method  is given in Algorithm 1. A few comments are in order:
\begin{itemize}
	\item The initial point $ \mathbb{X}^{0} =(\mathbf{X}^{0},r^{0}) $ must be an interior point, namely,  $ \mathbf{X}^{0}>0 $ and the initial parameter\\ $ r^{0}=<x^{0},z^{0}>/n$ has the correct order of magnitude.
	\item If $ \mathbf{X}^{k} \in \Xi, $ then $ (x^{k})^{-}=(z^{k})^{-}=0 $ and
	\begin{equation*} 
	\mathbf{d}^{k}=\left [
	\begin{array}{llllll} 
	d\mathbf{X}^{k} \\
	dr^{k}
	\end{array}
	\right ]=-\begin{pmatrix}
	\nabla_{x}H_{\theta_{1}}^{r} & \nabla_{z}H_{\theta_{1}}^{r} & \partial_{r}H_{\theta_{1}}^{r} \\
	\\
	0^{\text{T}}&0^{\text{T}}& \varepsilon+2r^{k}\\
	\end{pmatrix}^{-1}\left [
	\begin{array}{llllll} 
	H^{r}_{\theta_{1}}(\mathbf{X}^{k})\\
	\varepsilon  r^{k}+(r^{k})^{2}
	\end{array}
	\right ],
	\end{equation*}
	provided that the Jacobian matrix is invertible. The increment for the parameter is then
	\begin{equation*}
	dr^{k}=-\dfrac{\varepsilon r^{k}+(r^{k})^{2}}{\varepsilon+2r^{k}}.
	\end{equation*}	
	\item There is no need to truncate  the Newton direction $ \mathbf{d}^{k}$ to preserve positivity for $ x^{k+1}$ and $ z^{k+1},$ since nonnegativity is "guaranteed" at convergence. However, if we wish all the iterates are nonnegative, then we are free to carry out an additional damping after Step 4 (Armijo's line search).
\end{itemize}

\begin{proposition}
	Let  $ F $ be a continuous differentiable  $ \mathcal{P}_{0}$-function. Then, step 3 in Algorithm 1 is well-defined.
\end{proposition}
\begin{proof}
	From the update rule of Algorithm 1, we know that for all $ k \geq 0, ~r^{k}>0, \mathbf{X}^{k}>0,$ and $ \varepsilon>0, $  
	\begin{equation*}
	\text{det}~ \nabla \mathbb{H}_{\theta}(\mathbb{X}^{k})=(\varepsilon+2r^{k})~\text{det}~ \nabla H_{\theta_{1}}^{r^{k}}(\mathbf{X}^{k}).
	\end{equation*}
	In view of Theorem 3.11, we know that    $  \nabla H_{\theta_{1}}^{r^{k}}(\mathbf{X}^{k}) $ is nonsingular. Thus Step 3 of Algorithm 1 is well-defined.
\end{proof}

\subsection{Global convergence analysis}
\begin{definition}
	
	(Regular zero). Let $ \mathbb{X}^{*} \in \R^{2n+1}$ be a zero of $ \mathbb{H}_{\theta},$ that is, $ \mathbb{H}_{\theta}(\mathbb{X}^{*})=0.$ If the Jacobian matrix $\nabla_{\mathbb{X}}\mathbb{H}_{\theta}(\mathbb{X}^{*}) $ is nonsingular, $  \mathbb{X}^{*} $ is said to be a regular zero of $ \mathbb{H}_{\theta}. $ 
\end{definition}
The main interest of Algorithm 1 lies in the prospect of global convergence, as envisioned by the theory that we are developing now. This global convergence theory, is primarily based on the regularity of zeros [Definition 5.2]. We reproduce a concise result that can be found in the book of Bonnans  \cite{Bonnans}, in view of its importance to our algorithm.\\
We will prove the global convergence of Algorithme 1. First, we show that every $ \mathbf{d} \in  \triangle $ is a descent direction of $ \Theta $ at $ \mathbb{X}$, where
\begin{equation}
	\triangle(\mathbb{X})=\{\mathbf{d} \in \R^{n} ~~~|~~~ \nabla_{\mathbb{X}}\mathbb{H}_{\theta}(\mathbb{X})\mathbf{d}=\mathbb{H}_{\theta}(\mathbb{X})\}.
\end{equation}
\begin{lemma}(see \cite{Yamasshita})
	If $ \mathbb{X}$ is not a solution of NCP, i.e.  $ \Theta(\mathbb{X})>0 $, then every $ \mathbf{d} \in \triangle(\mathbb{X})$ satisfies the descent condition for $ \mathbb{X}$, i.e., $ \nabla \Theta(\mathbb{X})^{T}\mathbf{d}<0.  $
\end{lemma}
\begin{theorem}
	Every limit point $ \mathbb{X}^{*}=(\mathbf{X}^{*}, r^{*}) $ of a sequence$ \{\mathbb{X}^{k}\} $ generated by Algorithme 1 corresponds to a solution of NCP.  
\end{theorem}
\begin{proof}
Owing to Step 4, $ \{\Theta(\mathbb{X}^{k}) \} $ is decreasing monotonically nonnegative. It must converge to some $ \Theta(\mathbb{X}^{*})\geq 0$. We assume $ \Theta(\mathbb{X}^{*})   \geq 0$. Let $ \mathbb{X}^{*}$ be an accumulation point of $ \{\mathbb{X}^{*}\} $ and $ \{\mathbb{X}^{k}\}_{k} $ be a subsequence converging to 	$ \{\mathbb{X}^{*}\} $.\\
$ \Delta $ is uniformly compact near and closed at  $ \mathbb{X}^{*} $ (see \cite{Yamasshita}), we assume, without loss of generality, that \\ $\lim\limits_{\substack{k \rightarrow  \infty}}d^{k}=d^{*} \in \Delta(\mathbb{X}^{*})$. From Lemma 5.6, we wil get the contradiction if we can prove $ \nabla \Theta(\mathbb{X}^{*})^{T}\mathbf{d}^{*}=0. $ This can be obtained by considering the following two cases: 
\begin{itemize}
\item Suppose that $ \inf\{\zeta^{k}\} \geq \zeta > 0.$ Then we have 
\begin{equation*}
\Theta(\mathbb{X}^{k}+\zeta^{k}\mathbf{d}^{k})-\Theta(\mathbb{X}^{k})\leq \zeta^{k} \tau~ \nabla \Theta(\mathbb{X}^{k})^{T}\mathbf{d}^{k} \leq 0.
\end{equation*}
It is obvious that $ \nabla \Theta(\mathbb{X}^{*})^{T}\mathbf{d}^{*}=0  $ is satisfied.
\item Suppose that $ \inf \{\zeta^{k}\}=0.$ In this case, we assume $\lim\limits_{\substack{k \rightarrow  \infty}}\zeta^{k}=0 \in \Delta(\mathbb{X}^{*})$ without loss of generality. By line search, we have 

\begin{equation*}
	\dfrac{\Theta(\mathbb{X}^{k}+\frac{\zeta^{k}}{\varrho^{j_{k}}}\mathbf{d}^{k})-\Theta(\mathbb{X}^{k})}{\frac{\zeta^{k}}{\varrho^{j_{k}}}}> \tau~ \nabla \Theta(\mathbb{X}^{k})^{T}\mathbf{d}^{k},
\end{equation*}
taking the limit yields 

\begin{equation*}
	\nabla \Theta(\mathbb{X}^{*})^{T}\mathbf{d}^{*} \geq \tau \nabla \Theta(\mathbb{X}^{*})^{T}\mathbf{d}^{*}.
\end{equation*}
Since $ \tau \in (0, 1/2)$, we have $ \nabla \Theta(\mathbb{X}^{*})^{T}\mathbf{d}^{*} \geq 0. $ Hence $ \nabla \Theta(\mathbb{X}^{*})^{T}\mathbf{d}^{*} =0. $ \\
\end{itemize}
 We get the contradiction. The proof is complete.
\end{proof}
Below is a result about the Jacobian matrix of $ \mathbb{H}_{\theta}(\mathbb{X}),$ when $ r $ goes to $ 0$.
 \begin{lemma}
 Let $\mathbb{H}_{\theta}(\mathbb{X})  $ define by \eqref{889}. Then the Jacobian matrix of $\mathbb{H}_{\theta}(\mathbf{X}^{*},r)$ when $ r $ goes to $ 0 $ is:
 \begin{equation*}
 \lim_{r \to 0}  ~ (\nabla_{\mathbb{X}}	\mathbb{H}_{\theta}(\mathbf{X}^{*},r))= \begin{bmatrix}
 \nabla F(x^{*}) & -I_{n \times n}&0_{n \times 1}  \\
 \phi(Z^{*}) & \phi(X^{*})& 0_{n \times 1} \\
 0_{1 \times n}& 0_{1 \times n}      &      \varepsilon  \\
 \end{bmatrix},
 \end{equation*}
 where 
 \begin{equation*}
 \phi(Z^{*})_{ii}=\left\{\begin{array}{ccc}
 1 &\text{if}& z^{*}_{i} \neq 0 ~~\text{and} ~~ x^{*}_{i}=0\\
 0 &\text{if}& z^{*}_{i} = 0 ~~\text{and} ~~ x^{*}_{i} \ne 0,\\
 \frac{1}{4} &\text{if}& z^{*}_{i} = 0 ~~\text{and} ~~ x^{*}_{i}=0,\\
 \end{array} \right. \quad \text{and} \quad \phi(X^{*})_{ii}=\left\{\begin{array}{ccc}
 1 &\text{if}& x^{*}_{i} \neq 0 ~~\text{and} ~~ z^{*}_{i}=0\\
 0 &\text{if}& x^{*}_{i} = 0 ~~\text{and} ~~ z^{*}_{i} \ne 0,\\
 \frac{1}{4} &\text{if}& z^{*}_{i} = 0 ~~\text{and} ~~ x^{*}_{i}=0,\\
 \end{array} \right.
 \end{equation*}
 \end{lemma}
 
 \begin{proof}
 Let 
 \begin{equation*} 
 \mathbb{H}_{\theta}(\mathbb{X})=\left [
 \begin{array}{llllll} 
 \mathbb{H}_{\theta,~1}(\mathbb{X}) \\
 \mathbb{H}_{\theta,~2}(\mathbb{X}) \\
 \mathbb{H}_{\theta,~3}(\mathbb{X})
 \end{array}
 \right ]=\left [
 \begin{array}{llllll} 
 F(x)-z \\
 \dfrac{xz-r^{2}}{x+z+2r} \vspace{0.2cm}\\
 
 \frac{1}{2} \Vert x^{-}\Vert^{2}+ \frac{1}{2} \Vert z^{-}\Vert^{2}+r^{2}+\varepsilon r
 \end{array}
 \right ].
 \end{equation*}
 The Jacobian matrix of  $ \mathbb{H}_{\theta} $ is:  
 \begin{equation*}
 \nabla_{\mathbb{X}}	\mathbb{H}_{\theta}(\mathbb{X})=
 \begin{pmatrix}
 \nabla_{x} F(x)& -I_{n \times n} & 0_{n \times 1}  \\
 \nabla_{x} \mathbb{H}_{\theta,2}(\mathbb{X})	& \nabla_{z} \mathbb{H}_{\theta,2}(\mathbb{X})& \partial_{r} \mathbb{H}_{\theta,2}(\mathbb{X}) \\
 (x^{-})^{\text{T}}&(z^{-})^{\text{T}}& 2r+\varepsilon  \\
 \end{pmatrix}.
 \end{equation*}
 Let us to calculate $ \lim\limits_{\substack{r \rightarrow 0}}\nabla_{\mathbb{X}}	\mathbb{H}_{\theta}(\mathbf{X}^{*},r): $
 \begin{enumerate}
 	\item The derivative of $ \mathbb{H}_{\theta,~2}(\mathbf{X},r) $ with respect to $ x$  is:
 	\begin{equation*}
 	\nabla_{x} \mathbb{H}_{\theta,2} (x^{*},z^{*},r)=\text{diag}\left(\left(\dfrac{z^{*}+r}{x^{*}+z^{*}+2r}\right)^{2}\right)_{n \times n}, 
 	\end{equation*}
 	when $ r $ goes to $ 0 $  the only three cases to consider are: 
 	\begin{itemize}
 		\item  $ x_{i}^{*}\to 0, $ and $ z_{i}^{*}>0 $ $\forall i \in \{1, ..., n\} $ then\\
 		$ ~~~~~$   $ \lim\limits_{\substack{r \rightarrow 0 \\ x^{*}_{i}\rightarrow 0 \\ z^{*}_{i} \ne 0}} (\nabla_{x} 	\mathbb{H}_{\theta,2} (x^{*},z^{*},r))_{ii}=  \lim\limits_{\substack{r \to 0}} \left(\dfrac{z_{i}+r}{z_{i}+2r}\right)^{2}=1.$
 		\item  $ x_{i}^{*}>0, $ and $ z_{i}^{*} \to 0 $ $\forall i \in \{1, ..., n\} $ then \\ 
 		$ ~~~~~$   $ \lim\limits_{\substack{r \rightarrow 0 \\ z^{*}_{i}\rightarrow 0 \\ x_{i}^{*} \ne 0}} (\nabla_{x}\mathbb{H}_{\theta,2} (x^{*},z^{*},r))_{ii}=\lim\limits_{\substack{r \rightarrow 0 }}\left(\dfrac{r}{x_{i}+2r}\right)^{2}=0.$
 			\item  $ x_{i}^{*} \to 0, $ and $ z_{i}^{*} \to 0 $ $\forall i \in \{1, ..., n\} $ then \\ 
 		$ ~~~~~$   $ \lim\limits_{\substack{r \rightarrow 0 \\ z^{*}_{i}\rightarrow 0 \\ x_{i}^{*} \rightarrow 0}} (\nabla_{x}\mathbb{H}_{\theta,2} (x^{*},z^{*},r))_{ii}=\lim\limits_{\substack{r \rightarrow 0 }}\left(\dfrac{r}{2r}\right)^{2}=\dfrac{1}{4}.$
 	\end{itemize}
 	\item The derivative of  $ \mathbb{H}_{\theta,~2}(\mathbf{X},r) $ with respect to $ z$ is: 	
 	\begin{equation*}
 	\nabla_{z} \mathbb{H}_{s,2} (x^{*},z^{*},r)=\text{diag}\left(\left(\dfrac{x^{*}+r}{x^{*}+z^{*}+2r}\right)^{2}\right)_{n \times n},   
 	\end{equation*}
 	as below, the only three cases to consider are:
 	\begin{itemize}
 		\item  $ x_{i}^{*} \to 0, $ and $ z_{i}^{*}>0 $ $\forall i \in \{1, ..., n\} $ then \\
 		$ ~~~~~$   $ \lim\limits_{\substack{r \rightarrow 0 \\ x^{*}_{i}\rightarrow 0 \\ z_{i}^{*} \ne 0}} (\nabla_{z} 	\mathbb{F}_{s,2} (x^{*},z^{*},r))_{ii}=  \lim\limits_{\substack{r \rightarrow 0}}\left(\dfrac{r}{z_{i}+2r}\right)^{2}=0.$
 		\item  $ x_{i}^{*}>0, $ and $ z_{i}^{*} \to 0 $ $\forall i \in \{1, ..., n\} $ then \\ 
 		$ ~~~~~$   $ \lim\limits_{\substack{r \rightarrow 0 \\ z^{*}_{i}\rightarrow 0 \\ x_{i}^{*} \ne 0}}(\nabla_{z} 	\mathbb{F}_{s,2} (x^{*},z^{*},r))_{ii}=  \lim\limits_{\substack{r \rightarrow 0}} \left(\dfrac{x_{i}+r}{x_{i}+2r}\right)^{2}=1.$\\
 		\item  $ x_{i}^{*} \to 0, $ and $ z_{i}^{*} \to 0 $ $\forall i \in \{1, ..., n\} $ then \\ 
 		$ ~~~~~$   $ \lim\limits_{\substack{r \rightarrow 0 \\ z^{*}_{i}\rightarrow 0 \\ x_{i}^{*} \rightarrow 0}}(\nabla_{z} 	\mathbb{F}_{s,2} (x^{*},z^{*},r))_{ii}=  \lim\limits_{\substack{r \rightarrow 0}} \left(\dfrac{r}{2r}\right)^{2}=\dfrac{1}{4}.$
 	\end{itemize}
 	\item The derivative of  $ \mathbb{H}_{\theta,~2}(\mathbf{X},r) $  with respect to $ r$ is: 
 	\begin{equation*}
 	\partial_{r}\mathbb{H}_{\theta,~2}(x^{*},z^{*},r)=\text{diag}\left(\dfrac{-2r}{x+z+2r}+\dfrac{2(r^{2}-xz)}{(x+z+2r)^{2}}   \right)_{n \times 1},
 	\end{equation*}
 	when $ r $  goes to $ 0 $ the only three cases to consider are:
 	\begin{itemize}
 		\item $ x_{i}^{*} \to 0, $ and $ z_{i}^{*}>0 $  $\forall i \in \{1, ..., n\}$  then 
 		\begin{equation*}
 		\lim\limits_{\substack{r \rightarrow 0 \\ x^{*}_{i}\rightarrow 0 \\ z_{i}^{*} \ne 0}}(\partial_{r} 	\mathbb{H}_{\theta,~2} (x^{*},z^{*},r))_{i}
 		=\lim_{r \to 0}\left[\dfrac{-2r}{z_{i}+2r}+\dfrac{2r^{2}}{(z_{i}+2r)^{2}}\right]=0.
 		\end{equation*}	
 		\item  $ x_{i}^{*}>0, $ and $ z_{i}^{*} \to 0 $  $\forall i \in \{1, ..., n\}$ then 
 		\begin{equation*}
 		\lim\limits_{\substack{r \rightarrow 0 \\ z^{*}_{i}\rightarrow 0 \\ x_{i}^{*} \ne 0}}(\partial_{r} 	\mathbb{H}_{\theta,~2} (x^{*},z^{*},r))_{i}=\lim_{r \to 0}\left[\dfrac{-2r}{x_{i}+2r}+\dfrac{2r^{2}}{(x_{i}+2r)^{2}}\right]=0.
 		\end{equation*}
 			\item  $ x_{i}^{*} \to 0, $ and $ z_{i}^{*} \to 0 $  $\forall i \in \{1, ..., n\}$ then 
 		\begin{equation*}
 		\lim\limits_{\substack{r \rightarrow 0 \\ z^{*}_{i}\rightarrow 0 \\ x_{i}^{*} \rightarrow 0}}(\partial_{r} 	\mathbb{H}_{\theta,~2} (x^{*},z^{*},r))_{i}=\lim_{r \to 0}\left[\dfrac{-2r}{2r}+\dfrac{2r^{2}}{2r^{2}}\right]=0.
 		\end{equation*}
 	\end{itemize}
 \end{enumerate}
 \end{proof}

  We present now, three situations where we can conclude about the nonsingularity of $\lim\limits_{\substack{r \rightarrow 0}}\nabla_{\mathbb{X}}\mathbb{H}_{\theta}(\mathbf{X}^{*},r)$.\\

\begin{lemma}

	Suppose that $ \mathbf{X}^{*}=(x^{*}, z^{*}) $ is a solution of NCP, we have three possibilities when computing the determinant of $ \mathbb{H}_{\theta} $ on  $ (x^{*}, z^{*}).$

\begin{itemize}
	\item If $ z_{i}^{*}=0$ and  $ x_{i}^{*} = 0 , \quad \forall i=1,. . .n,~~$ i.e. $ (x^{*},z^{*})=(0,0) $\\
	\begin{equation*}
	\lim_{r \to 0} \text{det} ~ (\nabla_{\mathbb{X}}	\mathbb{H}_{\theta}(\mathbf{X}^{*},r))= \left\vert \begin{bmatrix}
	\begin{pmatrix}
	\nabla F(x^{*}) & -I_{n \times n}  \\
	\phi(Z^{*}) & \phi(X^{*}) \\
	\end{pmatrix}
	&        
	\begin{matrix}
	0 \\[3mm]
	0 \\[3mm]
	\end{matrix}
	\\ 
	0\quad  \quad  \quad 0      &      \varepsilon  \\
	\end{bmatrix}\right\vert= \varepsilon \left\vert
	\left(\begin{array}{cc}
	\nabla F(x^{*})&-I_{n \times n}\\
	\frac{1}{4}I_{n \times n}& \frac{1}{4}I_{n \times n}
	\end{array}\right) \right 
	\vert= \varepsilon \vert \dfrac{1}{4} \nabla F(x^{*})+\dfrac{1}{4} I_{n \times n} \vert ,
	\end{equation*}
	therefore the matrix $\lim\limits_{\substack{r \rightarrow 0}}\nabla_{\mathbb{X}}\mathbb{H}_{\theta}(\mathbf{X}^{*},r)$ exists   and is invertible  if $ F $ is $ \mathcal{P}_{0}$-function.

 \item  If $ \exists~ \alpha >0 $  such that $ x^{*}_{i}+z^{*}_{i}>\alpha$, $\forall i \in \{1, ..., n\}$, then from Lemma \eqref{22786633}, $ \forall ~~\I \subset \{ 1, ..., n\} $we have:  
 \begin{equation*}
 \lim_{r \to 0} \text{det} ~ (\nabla_{\mathbb{X}}	\mathbb{H}_{\theta}(\mathbf{X}^{*},r))= \left\vert \begin{bmatrix}
 \begin{pmatrix}
 \nabla F(x^{*}) & -I_{n \times n}  \\
 \phi(Z^{*}) & \phi(X^{*}) \\
 \end{pmatrix}
 &        
 \begin{matrix}
 0 \\[3mm]
 0 \\[3mm]
 \end{matrix}
 \\ 
 0\quad  \quad  \quad 0      &      \varepsilon  \\
 \end{bmatrix}\right\vert= \varepsilon \left\vert
 \left(\begin{array}{cc}
 \nabla F(x^{*})&-I_{n \times n}\\
 \phi(Z^{*})&  \phi(X^{*})
 \end{array}\right) \right 
 \vert= \varepsilon \vert \nabla F(x^{*})_{ \I\I} \vert,
 \end{equation*}
therefore the matrix $\lim\limits_{\substack{r \rightarrow 0}}\nabla_{\mathbb{X}}\mathbb{H}_{\theta}(\mathbf{X}^{*},r)$ exists   and is invertible  if $ F $ is $ \mathcal{P}$-function. \\

\item If $0 \leq x \perp z \geq 0~~$  (i.e. without the assumption of a strict complementarity) we have: 
 \begin{equation*}
\lim_{r \to 0} \text{det} ~ (\nabla_{\mathbb{X}}	\mathbb{H}_{\theta}(\mathbf{X}^{*},r))= \left\vert \begin{bmatrix}
\begin{pmatrix}
\nabla F(x^{*}) & -I_{n \times n}  \\
\phi(Z^{*}) & \phi(X^{*}) \\
\end{pmatrix}
&        
\begin{matrix}
0 \\[3mm]
0 \\[3mm]
\end{matrix}
\\ 
0\quad  \quad  \quad 0      &      \varepsilon  \\
\end{bmatrix}\right\vert= \varepsilon \left\vert
\left(\begin{array}{cc}
\nabla F(x^{*})&-I_{n \times n}\\
\phi(Z^{*})&  \phi(X^{*})
\end{array}\right) \right 
\vert=\varepsilon \left\vert
\left(\begin{array}{cc}
\nabla F(x^{*})-I_{n \times n}&-I_{n \times n}\\
\phi(Z^{*})+\phi(X^{*})&  \phi(X^{*})
\end{array}\right) \right 
\vert.
\end{equation*}
From Lemma 3.10, we take:
\begin{equation}
\begin{split}
 &M=\nabla F(x^{*})-I_{n \times n},\\
 &N_{t}=\phi(X^{*}),\\
 &N_{s}=\phi(Z^{*})+\phi(X^{*}),
\end{split}
\end{equation}
 where $ N_{s} $ is a positive diagonal matrix, and $ N_{t}$ is a nonnegative diagonale matrix. Therefore the matrix $\lim\limits_{\substack{r \rightarrow 0}}\nabla_{\mathbb{X}}\mathbb{H}_{\theta}(\mathbf{X}^{*},r)$ exists   and is invertible  if $ \nabla F(x^{*}) -I_{n \times n} $ is $ \mathcal{P}_{0}$-matrix.
 
\end{itemize}
\end{lemma}
In the following, we focus our attention on the superlinear convergence rate of Algorithme 1.
\begin{theorem}
	(Theorem 6.9, \cite{Bonnans}). Let $ \mathbb{H}_{\theta}: \R^{2n+1} \to \R^{2n+1} $ be a continuously-differentiable function.\\
	\begin{enumerate}[label=(\roman*)]
		\item (Local analysis) Let $ \mathbb{X}^{*} $ be a regular zero of $  \mathbb{H}_{\theta} $. If  ~$~ \mathbb{X}^{0}$ is close enough to $ \bar{\mathbb{X}}$, then $ \zeta^{k}=1 $ for all $ k,$ and $  \mathbb{X}^{k}$  converge to  $ \mathbb{X}^{*}$ super-linearly (and we recover the standard Newton method).\\
		\item (Limit point) Let $ \mathbb{X}^{*} $ be a limit point of sequence $ \{\mathbb{X}^{k}\}.$  If  $~\nabla \mathbb{H}_{\theta}( \mathbb{X}^{*}) $  is invertible, then  $ \mathbb{X}^{*}$ is a regular zero of $ \mathbb{H}.$ If $ \mathbb{X}^{*} $ is a regular zero of  $ \mathbb{H}_{\theta}$, then $ \zeta^{k}=1 $ for $ k $ big enough and   $  \mathbb{X}^{k}$  converge to  $ \mathbb{X}^{*}$ super-linearly.\\
	\end{enumerate}
\end{theorem}
\begin{proof}
We apply Lemma 5.4 and Lemma 5.6 under some condition on $ F $.
\end{proof}
The next lemma measures the "additional coercivity" effect of the smoothing.
\begin{lemma}
	Assume $ F $ is a $ \mathcal{P}_{0}$-function, then 
	\begin{enumerate}[label=(\roman*)]
		\item  $  \mathbb{H}_{\theta} $ is a $ \mathcal{P}$-function.
		\item If  $  \mathbb{H}_{\theta}(\mathbf{X},0) $ exists then it is a $ \mathcal{P}_{0}$ function.
	\end{enumerate}
\end{lemma}
\begin{proof}
	(i) Let $ X, ~Y $ be two distinct vectors of $ \R^{2n} $. Since $ F $ is a $  \mathcal{P}_{0} $ function there exists an index $ i\in \{1, ...,2n \}$ such that $ X_{i} \ne Y_{i} $ and $ (X_{i}-Y_{i})(F_{i}(X)-F_{i}(Y))\geq 0. $ Without loss of generality, we can suppose that $ X_{i}>Y_{i} $ and $ F_{i}(X)>F_{i}(Y).$ \\
	Since $ \psi $ and $ \psi^{-1} $ are decreasing we obtain consecutively that for any $ r>0,$ 
	\begin{equation} \label{712}
	\psi(X_{i}/r)+\psi(F_{i}(X)/r)<\psi(Y_{i}/r)+\psi(F_{i}(Y)/r),
	\end{equation}
	so 
	\begin{equation*}
	G_{r}(X_{i},F_{i}(X))>	G_{r}(Y_{i},F_{i}(Y)).
	\end{equation*}
	Hence,  $  \mathbb{H}_{\theta} $ is a $  \mathcal{P} $ function. \\
	We will now deal with the case where $ X_{i}=Y_{i}, ~~ \forall i <2n+1.$ For $ i=2n+1, $ we can suppose that $ X_{2n+1}>Y_{2n+1} $ 
	\begin{equation*}
	\begin{split}
	(X_{2n+1}-Y_{2n+1})( \mathbb{H}_{\theta}(X)_{n+1}-\mathbb{H}_{\theta}(Y)_{n+1}=&(r_{1}-r_{2})\left(\frac{1}{2} \Vert x^{-}\Vert^{2}+\frac{1}{2} \Vert z^{-}\Vert^{2}+r_{1}^{2}+ \varepsilon r_{1}-\frac{1}{2} \Vert x^{-}\Vert^{2}-\frac{1}{2} \Vert z^{-}\Vert^{2}-r_{2}^{2}- \varepsilon r_{2}\right)\\
	& (r_{1}-r_{2})(r_{1}^{2}+\varepsilon r_{1}-r_{2}^{2}-\varepsilon r_{2})>0
	\end{split}
	\end{equation*}
	Hence,  $  \mathbb{H}_{\theta} $ is a $  \mathcal{P} $ function.\\
	(ii) If $  \mathbb{H}_{\theta}(\mathbb{X},0) $ exists, passing to the limit in \eqref{712} as $ r \searrow 0, $ we obtain that $  \mathbb{H}_{\theta}(\mathbb{X},0) $ is a $ \mathcal{P}_{0}$ function.
\end{proof}
Now we would like to study the asymptotic behavior of the Jacobian matrix of our method with the Jacobian matrix of interior-point methods when $  r  $ goes to $ 0 $ and we need a lemma that is used to prove our main result.
\begin{lemma} \label{22786633}
	We consider the following system 
	\begin{equation} 
	\begin{array}{llllll}
	Z.X=0 \\
	Z \geq 0,~~ X\geq 0 ,
	\end{array}
	\end{equation}		
	where $ Z=\text{diag}(z) $ and  $ X=\text{diag}(x)$.\\ Assume that $ Z,~X$ are strictly complementary (i.e. $ \exists~ \alpha >0 $ such that $ Z+X>\alpha$). Then $ J $ is singular if and only if $ T $ is singular, where 
	\begin{equation*}
	J=\left(\begin{array}{cc}
	\nabla F(x)&-I\\
	Z& X
	\end{array}\right), \quad \text{and} \quad  
	T=\left(\begin{array}{cc}
	\nabla F(x)&-I\\
	\phi(Z)& \phi(X)
	\end{array}\right),
	\end{equation*}
 where $\phi(.)$ is defined in Lemma 5.6, here $ \phi $ operates componentwise on $ Z $ (resp. on $ X $).
\end{lemma}
\begin{proof}
	By the strict complementarity hypothesis, we range the rows and the columns of  $ J $ and  $ T$ as follows   
	\begin{displaymath}
	J_{\sigma}=\left(\begin{array}{cc}
	\nabla F(x)_{\sigma}&-I_{\sigma}\\
	\left(\begin{array}{cc}
	Z_{1}&0\\
	0& 0
	\end{array}\right)& \left(\begin{array}{cc}
	0&0\\
	0& X_{2}
	\end{array}\right)
	\end{array}\right),
	\end{displaymath}
	where $ X_{2} > 0$ and $ Z_{1} >0,$ and 
	\begin{displaymath}
	(T)_{\sigma}=\left(\begin{array}{cc}
	\nabla F(x)_{\sigma}&-I_{\sigma}\\
	\left(\begin{array}{ccc}
	\begin{array}{ccc}
	1&        & \\
	& \ddots &  \\
	&        & 1 \\
	\end{array}&        & 0\\
	&  &  \\
	0 &        & 0 \\
	\end{array}\right)& \left(\begin{array}{ccc}
	
	0&        & 0\\
	& \ddots &  \\
	0&        & \begin{array}{ccc}
	
	1&        & 0\\
	& \ddots &  \\
	0&        & 1 \\
	\end{array} \\
	\end{array}\right)
	\end{array}\right).
	\end{displaymath}
	The determinant of the two matrices  $ J_{\sigma} $ and $ (T)_{\sigma} $ are equal to 
	\begin{equation*}
	\text{det}(J_{\sigma})= \left\vert \begin{array}{cc}
	\nabla F(x)_{\sigma}&-I_{\sigma}\\
	\left(\begin{array}{cc}
	Z_{1}&0\\
	0& 0
	\end{array}\right)& \left(\begin{array}{cc}
	0&0\\
	0& X_{2}
	\end{array}\right)
	\end{array}\right \vert =\pm \prod_{i \in \I_{1}}^{} x_{i} \prod_{i \in \I_{2}}^{} z_{i}~~ \text{det}(C),
	\end{equation*}
	\begin{equation*}
	\text{det}(T_{\sigma})=\left \vert\begin{array}{cc}
	\nabla F(x)_{\sigma}&-I_{\sigma}\\
	\left(\begin{array}{ccc}
	\begin{array}{ccc}
	1&        & \\
	& \ddots &  \\
	&        & 1 \\
	\end{array}&        & 0\\
	&  &  \\
	0 &        & 0 \\
	\end{array}\right)& \left(\begin{array}{ccc}
	0&        & 0\\
	& \ddots &  \\
	0&        & \begin{array}{ccc}
	1&        & 0\\
	& \ddots &  \\
	0&        & 1 \\
	\end{array} \\
	\end{array}\right)
	\end{array}\right\vert=\pm \prod_{i \in \I_{1}}^{} \phi(x_{i}) \prod_{i \in \I_{2}}^{} \phi(z_{i})~~ \text{det}(C),
	\end{equation*}
	where $ C $ is a certain matrix, $ \I_{1}= \{i ~~|~~ x_{i}> 0\}$  and $ \I_{2}= \{i ~~|~~ z_{i}> 0 \}$. Since
	\begin{equation*}
	\pm \prod_{i \in \I_{1}}^{} x_{i} \prod_{i \in \I_{2}}^{} z_{i} ~~~\quad\text{and} ~~~ \quad\prod_{i \in \I_{1}}^{} \phi(x_{i}) \prod_{i \in \I_{2}}^{} \phi(z_{i}),
	\end{equation*}
	are nonzeros, then we can conclude that $ J $ and $ T $ are invertibles and singulars at the same time.
\end{proof}
\begin{theorem} \label{765699}
	Suppose that $ \mathbf{X}^{*}=(x^{*}, z^{*}) $ is a solution of NCP which satisfies the strict complementarity (i.e. $ \exists~ \alpha >0 $  such that $ x^{*}_{i}+z^{*}_{i}>\alpha$, $\forall i \in \{1, ..., n\} $), and    $ \nabla H_{0}(\mathbf{X}^{*}) $ define by \eqref{0854356}  (the Jacobian matrix of the Interior-Point Method) is invertible. Then  $ \lim\limits_{\substack{r \rightarrow 0}}\nabla_{\mathbb{X}}	\mathbb{H}_{\theta}(\mathbf{X}^{*},r)  $ is invertible, i.e. the two Jacobian matrices are singular or nonsigular at the same time.
\end{theorem}
\begin{proof}
	
	In view of Lemma 5.6, and thanks to the assumption $ \mathbf{X}^{*}=(x^{*},z^{*}) $ is a solution of NCP, we have $ x^{*}\geq 0 $ and  $ z^{*}\geq 0 $, so that $ x^{-}=z^{-}=0. $ Hence 
	\begin{equation*}
	\lim_{r \to 0} \text{det} ~ (\nabla_{\mathbb{X}}	\mathbb{H}_{\theta}(\mathbf{X}^{*},r))= \left\vert \begin{bmatrix}
	\begin{pmatrix}
	\nabla F(x^{*}) & -I_{n \times n}  \\
	\phi(Z^{*}) & \phi(X^{*}) \\
	\end{pmatrix}
	&        
	\begin{matrix}
	0 \\[3mm]
	0 \\[3mm]
	\end{matrix}
	\\ 
	0\quad  \quad  \quad 0      &      \varepsilon  \\
	\end{bmatrix}\right\vert= \varepsilon \left\vert
	\left(\begin{array}{cc}
	\nabla F(x^{*})&-I_{n \times n }\\
	\phi(Z^{*})& \phi(X^{*})
	\end{array}\right) \right 
	\vert,
	\end{equation*}
	where $\phi(.)$ is defined in Lemma 5.6,~ $ Z^{*}=\text{diag}(z^{*}) $ and $ X^{*}=\text{diag}(x^{*})$.\\
	From Lemma \eqref{22786633}, we conclude that if  $ \nabla_{\mathbf{X}}H_{0}(\mathbf{X}^{*})$ is invertible  then $\lim\limits_{\substack{r \rightarrow 0}}\nabla_{\mathbb{X}}\mathbb{H}_{\theta}(\mathbf{X}^{*},r)  $ is invertible. This means, that if the Interior Point Method converges our method converges.\\	
Hypothesis $ H_{0}~ (F ~ \text{is a} ~ \mathcal{P}_{0}\text{-function}) $ assures us that our method is well defined and the Theorem 5.10 shows that the domain of convergence of our method is at least as large as that of the interior-point methods. \end{proof}
\section{Numerical Experiments and Applications}
In this section, we present some numerical experiments for the two smoothing functions. Our aim is just to verify the theoretical assertions for these two "extreme" cases.\\
 First, we study eight test problems with various sizes and characteristics. Then, we present a comparison on some randomly generated problems of our method and other approaches that have been suggested recently in \cite{El Ghami, FB}. We also present numerical results for  a concrete example. All the codes are written in MATLAB 2020R and run in the system of Windows 10 with PC i5 8-th Gen and 16.00 GB RAM. We take the precision $ \varepsilon=10^{-9}$ (the termination criterion). If the number of iteration is more than $ 1000 $, or the obtained step size at some iteration is less than $ 10^{-5}$, we consider that the algorithm fails.
\begin{example}
We consider eight test problems (that can be found in \cite{18, 19, 20, 22, 23}) with various sizes and characteristics. In some cases,  $F$ is monotone or strongly monotone whereas others can have a non-connected solution set, in this case, $ F $ is at most a $\mathcal{P}_{0}$ function.\\
 A precise description of each test problem is given in the appendix. We present in the appendix the numerical results obtained by the well-known projection iterative method (see  \cite{24}) when it is used exactly in the same conditions.
\begin{table}[H]
	\caption{Results for $ \theta_{1} $ and $ \theta_{2} $ }
	\begin{center}
		\begin{tabular}{ccccccc}
			\hline
			\hline
			  Pb & size &   Iter $  $&  Opt.& Feas.& cpu time(s) &r \\
			  & & $ (\theta_{1},~\theta_{2}) $&$ (\theta_{1},~\theta_{2}) $&$ (\theta_{1},~\theta_{2}) $&$ (\theta_{1},~\theta_{2}) $&$ (\theta_{1},~\theta_{2}) $\\
			\hline
			\hline 
			P1  &  10 & (114,  47)  &(7.58e-08, 7.66e-10) & (0, 8.23e-08 )  & (0.04, 0.0073)& (2.96e-04, 0.0014) \\
			\hline
			  &    100  & (134, 65 ) &  (9.31e-09,  1.15e-12) &  (0, 1.51e-09)  & (0.14, 0.07)&(1.06e-04, 9.21e-04)  \\
			  \hline
			 &  500 & (148,  66)  & (1.87e-09, 6.39e-13)&  (0, 4.28e-09)   &  (5.24, 2.52)& (4.90e-05, 8.95e-04) \\
			 \hline
			 & 1000& (153, 68) &  (1.05e-09, 7.82e-11) & (0, 1.04e-06) &(25.30, 11.41)&(3.70e-05, 0.0012)  \\
			 \hline
			 \hline
			 P2  &  10 & (116, 47)  & (7.53e-08, 7.66e-10) & (0, 8.23e-08)   & (0.01, 0.02)& (2.83e-04, 1.41e-03)  \\
			 \hline
			 &    100  & (133, 74) &  (7.64e-09, 2.64e-10) & (0, 3.47e-07)  & (0.19, 0.10 )&(8.74e-05, 0.0013) \\
			 \hline
			 &  500 & (147, 84)  &  (1.60e-09, 1.29e-09)& (0, 8.62e-06)   &  (5.25, 3.14)&(4.01e-05, 0.0014) \\
			 \hline
			 & 1000& (153, 115) &  (8.22e-10, 2.25e-12) & (0, 3.02e-08) & (24.47, 19.89)&(2.86e-05, 9.53e-04 )  \\
		 \hline
		 \hline
		P3  &  10 & (14, 16)  & (3.46e-07, 1.65e-16) & (0, 3.99e-15)  & (0.006, 0.02)& (0.001, 4.49e-04)  \\
		\hline
		&    100  & (108, 44) &  (4.99e-07, 3.63e-17) &  (0, 4.75e-18)  & (0.22, 0.05)& (0.0014, 0.0042)  \\
		\hline
		&  500 & (353,140)  &  (7.58e-07, 3.97e-08)& (9.68e-10, 7.54e-08)  &  (35.03, 5.14)&(0.0011, 0.002) \\
		\hline
		& 1000& (675, 265) &  (8.94e-07, 2.32e-07 ) &  (1.31e-09, 1.58e-08)& (91.41, 24.77)& (0.0011, 0.002 ) \\
		 \hline
		 \hline
		P4  &  4 & (53, 58)  & (2.97e-07, 1.20e-07) & (1.24e-04,  4.19e-05)   & (0.008, 0.083)& (6.06e-04, 1.72e-04)  \\
		\hline
		\hline
		P5&   4  & (16, 14) &  (3.02e-07, 1.92e-07) &  (0, 3.84e-07 )  &(0.003, 0.009)&(0.0026, 0.018)  \\
		\hline
		\hline
		P6&   7  & (10, 13) &  (1.06e-07, 7.16e-08) & (0, 0) & (0.1264, 0.0044)& (0.0016, 1.33e-04)  \\
		\hline
		\hline
			P7&   5  & (33, 30) &  (2.23e-08, 1.16e-07 ) &  (0, 3.44e-11)  & (0.011, 0,016)& (0.004, 0.003)  \\
			\hline
			\hline
				P8&   10  & (65, 45) &  (7.21e-08, 2.27e-08) &  (1.34e-09, 5.18e-08)  & (0.18, 0.16)& (0.0018, 3.76e-04)  \\
			\hline
			\hline 
		\end{tabular}
		\label{tab5}
	\end{center}
\end{table}
 In this table, $\mathbf{Size}$ stands for the number of variables, $\mathbf{Iter}  $ corresponds to the total number of jacobian evaluations, $\mathbf{Opt.}$ and $\mathbf{Feas.}$ correspond to the following optimality and feasibility measures
 \begin{equation*}
 \text{Opt.}:=\max_{1\leq i\leq n} |x_{i}F(x_{i})|\quad \text{and} \quad \text{Feas.}:= \|\min(x,~0)\|_{1}+\|\min(F(x),~0)\|_{1}.
 \end{equation*}
 The results clearly show that our methods are efficient, competitive and superior to the projection one (the results of the projection method are given in the appendix). We also remark that the second smoothing function is much more efficient and powerful than the first one.
\end{example}
\begin{example}
	This example is described in \cite{22, 23}. The coorresponding function $ F(x) $ is of the form:
	\begin{equation*}
	 F(x)=(AA^{T}+B+D)x+q,
	\end{equation*}
	where the matrices $ A,~B $ and $ D $ are randomly generated as: any entry of the square $ n\times n $ matrix $ A $ and of the $ n \times n $ skew-symmetric matrix $ B $ is uniformly generated from $ ] -5,~5 [,$ and any entry of the diagonal matrix $ D $ is uniformly generated  from $ ] 0,~3 [.$ The vector $ q $ is uniformly generated from $ ] -500,~0 [.$ \\
	The matrix $ AA^{T}+B+D $ is a positive definite and the function $ F $ is strongly monotone. We used the M-files proposed in \cite{22} to generate $ A,~B,~D $ and $ q.$\\
	In this example, we will compare our methods already mentioned in sections 3 and 4, named: Algorithm 1 ($ \theta_{1} $), Algorithm 1 ($ \theta_{2} $) to some other methods (Newton-Min method (Min-Alg), Fischer-Burmeister method \cite{FB} (FB-Alg) and the classical interior-point method \cite{El Ghami} (IPM-Alg)). In order to complete this experiment, we propose the performance profiles, developed by E. D. Dolan and J. J. Moré \cite{26}, as a tool for the comparative analysis of these methods. \\
	We set "$ n_{s}=5 $" as the number of methods and we have chosen "$ n_{p}=100$" (problems to be tested). We are interested in the comparison of the computation time and the number of iterations.

 \begin{figure}[H]\label{fig1}
	\begin{minipage}[H]{.46\linewidth}
		\begin{center}
			\includegraphics[width=8.5cm,height=6.5cm]{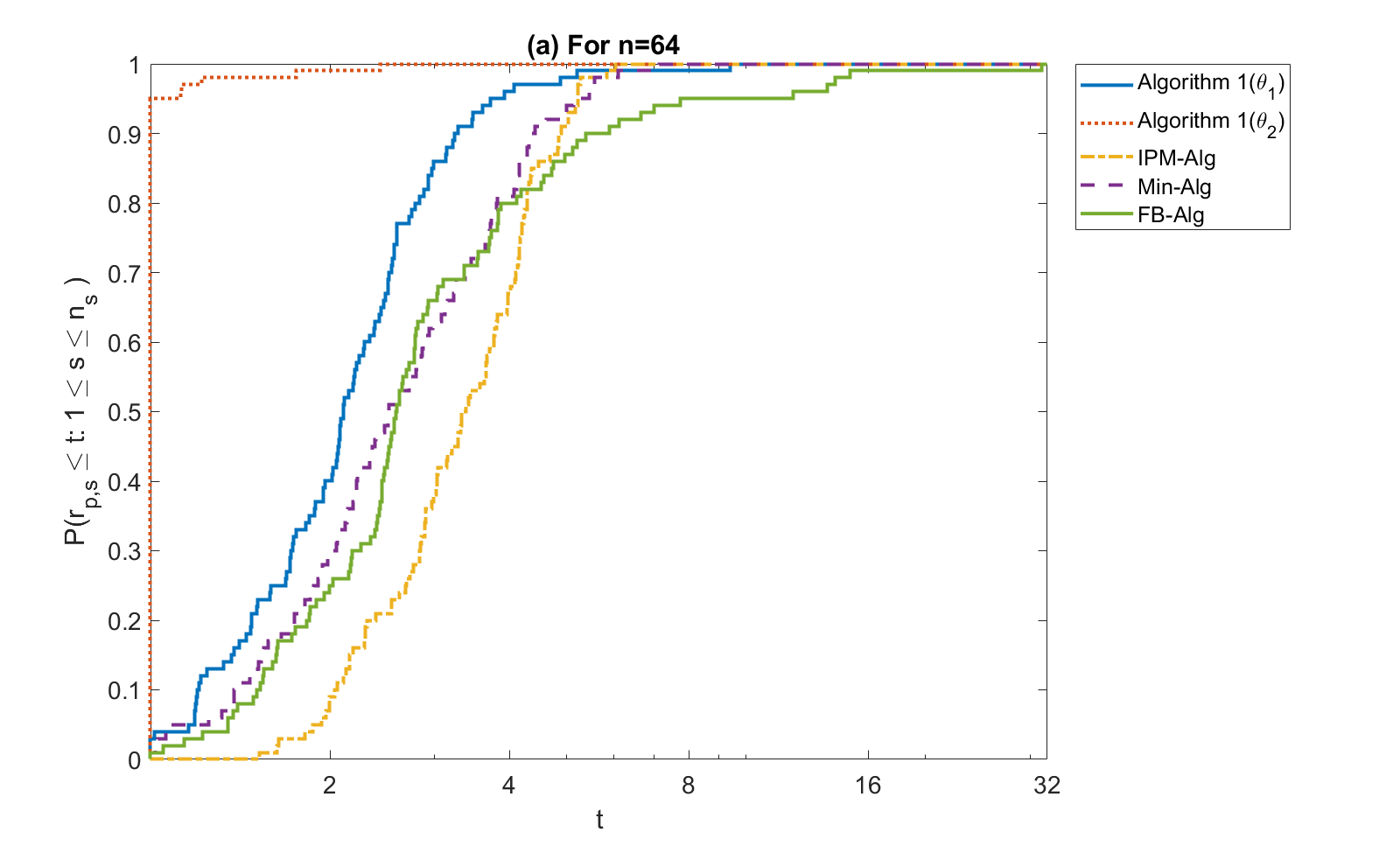}
		\end{center}
	\end{minipage} \hfill
	\begin{minipage}[H]{.46\linewidth}
		\begin{center}
			\includegraphics[width=8.5cm,height=6.5cm]{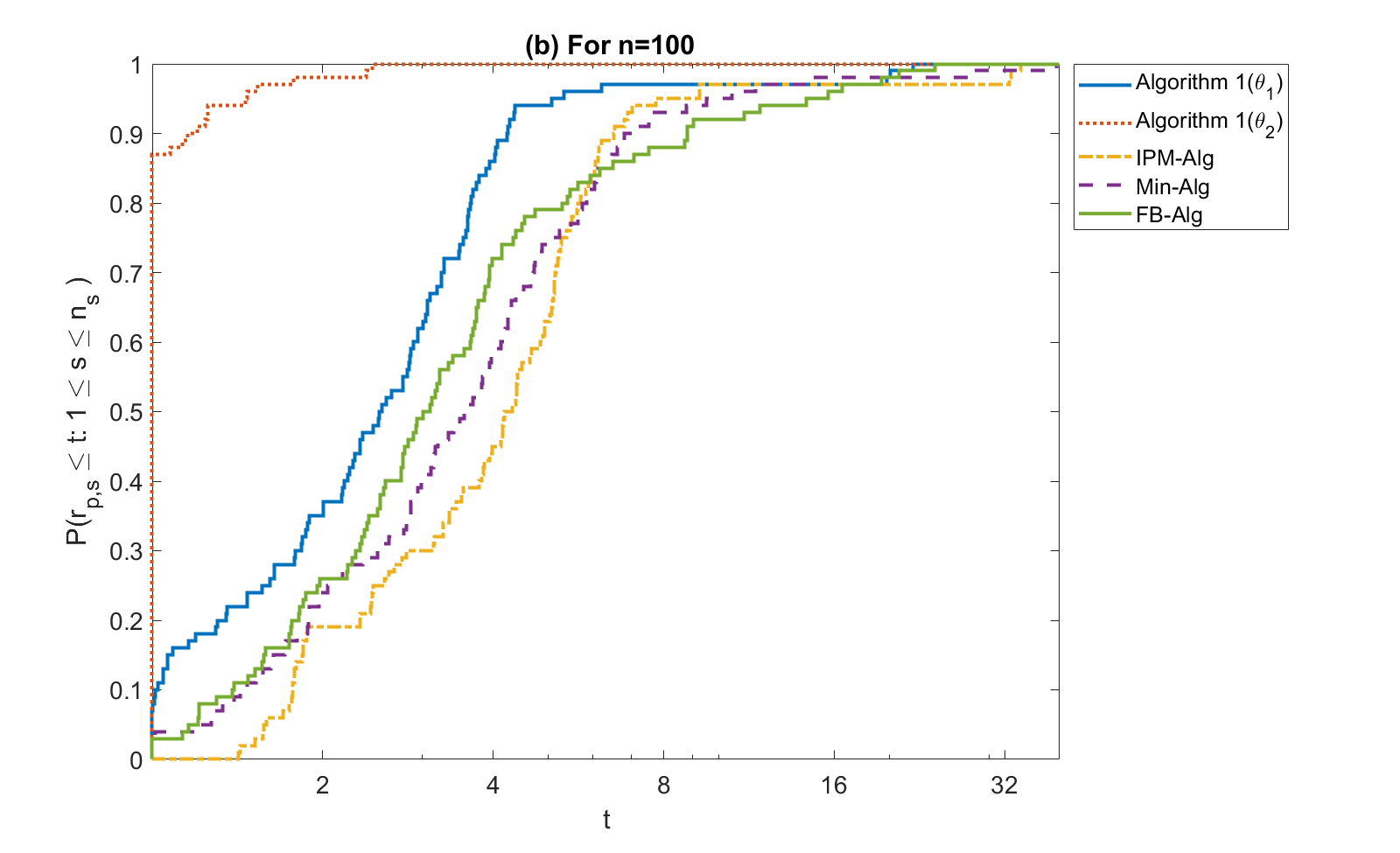}
		\end{center}
	\end{minipage}
	\caption{Performance profiles where $ t_{p,~s}$ represents the average computation time.}
\end{figure}
The figure above shows the performance profiles of five methods where the performance measure is execution time. It is clear that our method with the $ \theta_{2} $ function captures our attention (admits the highest probability value). In fact, in the interval [0, 1], our method is able to solve $ 99\% $ of the problems, while the other methods do not reach $ 20\% $  and require more time. We also notice that IPM-Alg is the slowest compared to others. However, for $ t> 4 $, the three algorithms FB-Alg, Min-Alg, and Algorithm 1 with $ \theta_{1} $ function confirm their robustness. Figure 3 also indicates that, with respect to the computation time, with the same initial points and under the same stopping criterion, our method with $ \theta_{2} $ (resp. $ \theta_{2} $ function ) is the fastest method, followed respectively by Min-Alg, FB-Algor, and IPM-Alg.
\begin{figure}[H]\label{fig1}
	\begin{minipage}[H]{.46\linewidth}
		\begin{center}
			\includegraphics[width=8.5cm,height=6.5cm]{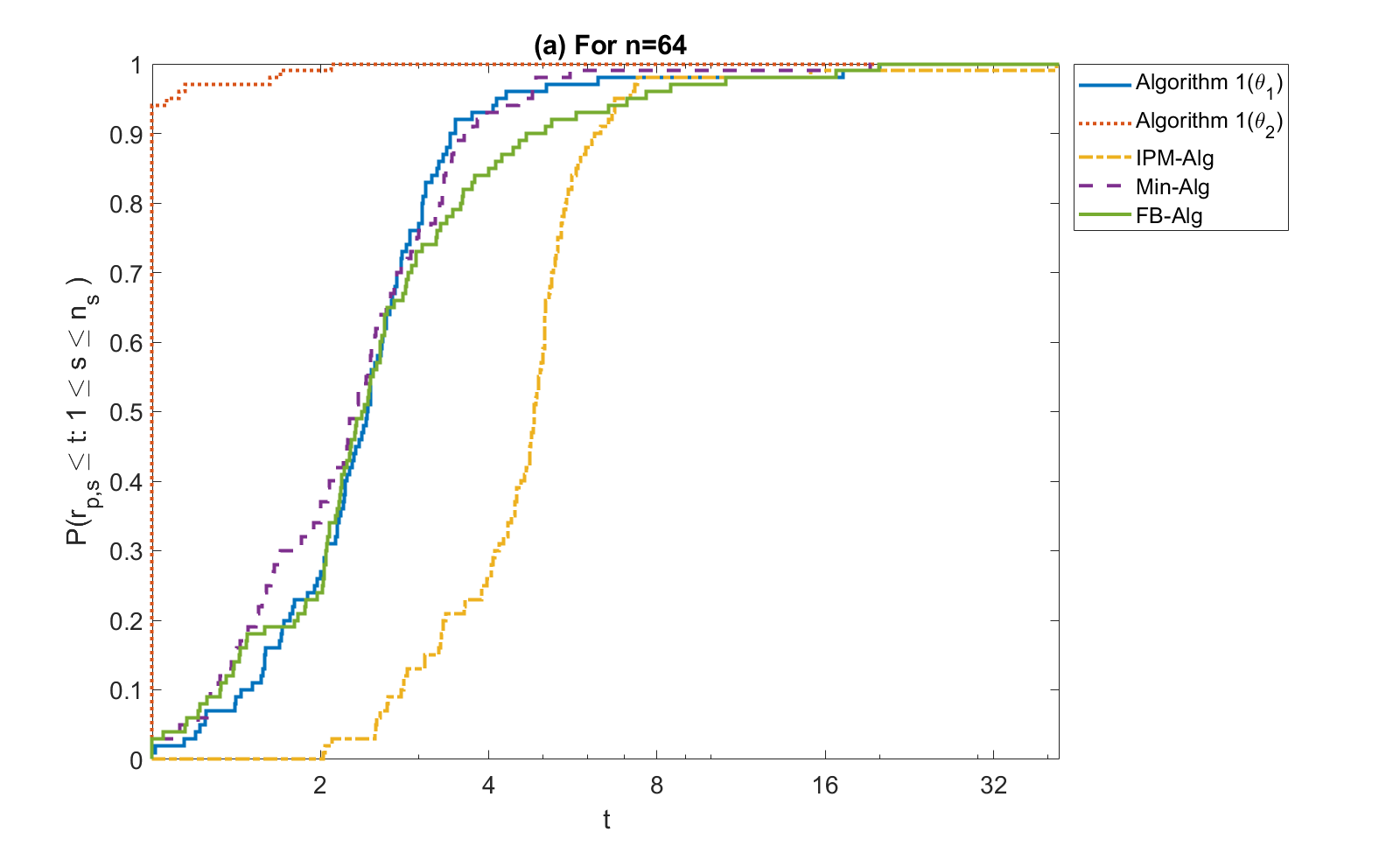}
		\end{center}
	\end{minipage} \hfill
	\begin{minipage}[H]{.46\linewidth}
		\begin{center}
			\includegraphics[width=8.5cm,height=6.5cm]{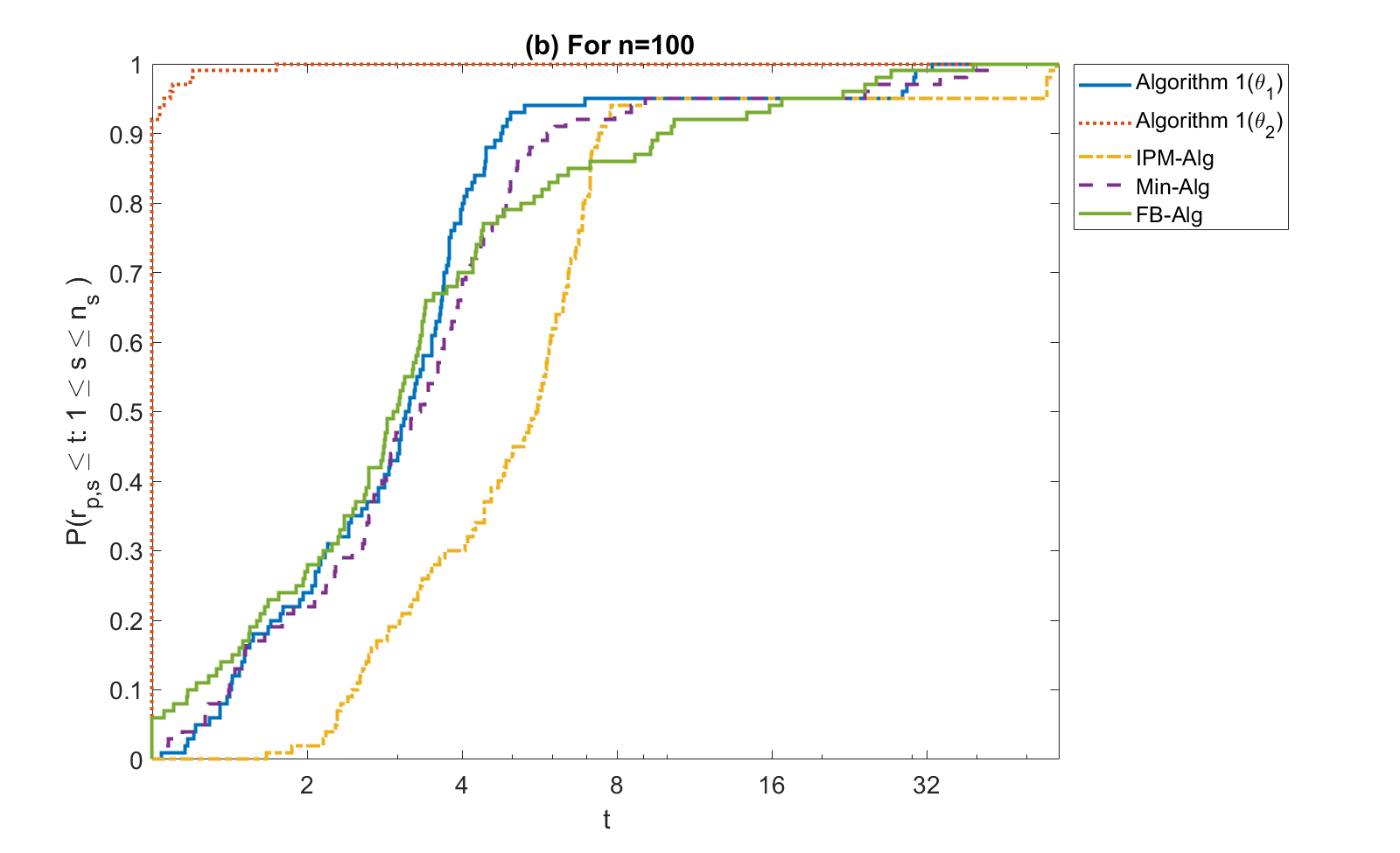}
		\end{center}
	\end{minipage}
	\caption{Performance profiles where $ t_{p,~s}$ represents the average number of iteratoins.}
\end{figure}
	
In Figure 4, we illustrate the performance profiles of five methods considering the number of iterations required as a performance measure. We notice that our method with the $ \theta_{2} $ function is the winner (admits the highest probability value) followed by our method with the $ \theta_{1} $ function, Min-Algo, and  FB-Alg. We also note that IPM-Alg needs more iterations to resolve problems. The performance of our method with the $ \theta_{1} $ function becomes interesting beyond $ t = 3 $.

\end{example}
\begin{example}
	(Geochemical Models \cite{25})
	The problem comes from Geochemistry. We introduce a model which are 2-salts. The main idea is that we need to find a way to reformulate a general problem to a problem which has a form like $ G(x)=0.$  We show the numerical results by applying several iteration methods.\\
	Let $ T, K $ are constant vectors which have meaning in chemistry. We define the problem as follows: \\
	let $ x=(x_{1}, x_{2}, x_{3})$ and $ p=(p_{1}, p_{2})$, 
	\begin{equation*}
	\begin{split}
	H: \R^{5} \to \R^{3}&\\
	(x,~p) \to H(x,p)&=\begin{pmatrix}
	T_{1}-x_{1}-p_{1} \\
	T_{2}-x_{2}-p_{2}\\
	x_{3}-x_{2}-x_{1} \\
	\end{pmatrix}
	\end{split}	
	\end{equation*}
	\begin{equation*}
	\begin{split}
	F: \R^{3} \to \R^{2}&\\
	x \to F(x)&=\begin{pmatrix}
	K_{1}-x_{1}x_{3} \\
	K_{2}-x_{2}x_{3}\\
	\end{pmatrix}
	\end{split}	
	\end{equation*}
		\begin{equation*}
	\begin{split}
	G: \R^{5} \to \R^{5}&\\
	(x,p)\to G(x,p)&=\begin{pmatrix}
H(x,p) \\
	P^{T}F(x)\\
	p \geq 0, F(x) \geq 0\\
	\end{pmatrix}.
	\end{split}	
	\end{equation*}
	We want to solve the equation 
	\begin{equation} \label{11123}
		G(x,p)=0 \iff \begin{pmatrix}
		H(x,p) \\
		P^{T}F(x)\\
		p \geq 0, F(x) \geq 0\\
		\end{pmatrix}=0.
	\end{equation}

Using our approach with $ \theta_{r}=\theta_{r}^{1}$ (resp.  $ \theta_{r}=\theta_{r}^{2}$ ), we reformulate \eqref{11123} and we get 

\begin{equation} 
G(x,p)= \begin{pmatrix}
K_{1}-x_{1}x_{3}-z_{1}\\
K_{2}-x_{2}x_{3}-z_{2}\\
T_{1}-x_{1}-x_{4} \\
T_{2}-x_{2}-x_{5}\\
x_{3}-x_{2}-x_{1}\\
G_{r}(x_{4},z_{1})\\
G_{r}(x_{5},z_{2})\\
\frac{1}{2} \Vert x^{-}\Vert^{2}+\frac{1}{2} \Vert z^{-}\Vert^{2}+ r^{2}+\varepsilon r\\
\end{pmatrix}=0.
\end{equation}
where $ G_{r}(x,z)=\dfrac{xz-r^{2}}{x+z+2r} $ with  $ \theta_{r}=\theta_{r}^{1}$    and $ G_{r}(x,z)=-r\log(e^{-x/r}+e^{-z/r}) $ with  $ \theta_{r}=\theta_{r}^{2}$, and considering $ x_{4}=p_{1},~ x_{5}=p_{2}.$\\
Since our focus is on the effect of different smoothing approaches in solving \eqref{11123}, we replace the complementarity constraint by the Min function, i.e.:
\begin{equation} 
p^{T}F(x)=0 \iff \min(p,F(x))=0,
\end{equation}
and by the Fischer-Burmeister function \cite{FB}, i.e.
\begin{equation} 
p^{T}F(x)=0 \iff \sqrt{p^{2}+F(x)^{2}}-(p+F(x))=0,
\end{equation}
respectively. The corresponding two algorithms are referred to as Min-Alg and FB-Alg, respectively. We make a comparison among Algorithm 1, Min-Alg, and FB-Alg, IPM-Alg (The classical interior-point method \cite{El Ghami}) by implementing these algorithms to solve problem \eqref{11123}.\\
The table 2, Figure 5 and  Figure 6 show the results with the intial point $ x_{0}=(3,1,4,5,6)^{T},~ T=(2,6)^{T},\\
~K=(37.5837, 7.6208)^{T}$. 
\begin{table}[H]
	\caption{Results for $ \theta_{1} $ and $ \theta_{2} $ }
	\begin{center}
		\begin{tabular}{cccccc}
			\hline
			\hline
			Iter & Algorithm 1 ($ \theta_{1} $) &   Algorithm 1 ($ \theta_{2} $) $  $&  Min-Alg& FB-Alg& IPM-Alg \\
			$ k $&$ \Vert G(x^{k}) \Vert  $ &$ \Vert G(x^{k}) \Vert  $ & $\Vert G_{Min}(x^{k}) \Vert $  &$\Vert G_{FB}(x^{k}) \Vert $&$\Vert G_{IPM}(x^{k}) \Vert $\\
			\hline
			\hline 
			 0 &   24.5837 &    24.5837 & 24.5837 & 24.5837  & 24.5837 \\
			 1&    8.2104  &  13.2212 &   13.3538 &   8.9798  &     8.3720 \\
			2&  2.5518 & 12.0213  & 12.1412&  4.2253   &  6.2486 \\
			3& 4.1723& 9.6610 &  11.0112 &  1.2774 &1.8581 \\
			4 &  1.9497 & 6.1377  & 9.9651 &  0.4355   & 1.2035  \\
		5	&    0.0382  &  1.3388 &  9.0020 & 0.2001  & 0.8359 \\
		6	&  0.0063 &  0.3347  &   8.1192& 0.1049  &  0.2406 \\
		7	&  0.0012 & 0.0838  &  7.1321& 0.0537   &  0.0032 \\
		8	&  2.4936e-04 & 0.0210  &  6.0534&  0.0272   &  5.1935e-05 \\
		9	&  5.0533e-05 & 0.0052 &  4.6609&  0.0137   &  1.6166e-05 \\
	10	&  1.0487e-05 & 0.0013  &   2.9998&  0.0068   &  4.9426e-06 \\
		11	& 2.1474e-06  & 3.2750e-04  &  1.2778& 0.0034  &  1.5108e-06 \\
		12	&  4.4258e-07  & 8.1875e-05  &  0.0563&  0.0017  &  4.6172e-07\\
		13	& 9.1435e-08  & 2.0469e-05  &  1.2410e-05& 8.5782e-04   &  1.4110e-07 \\
		14	& 1.8834e-08  & 5.1172e-06  &  4.4658e-12& 4.2899e-04   &  4.3119e-08 \\
			15	& 4.5031e-09  & 1.2793e-06  &  & 2.1451e-04  &  1.3176e-08 \\
			16	&  9.6186e-10  & 3.1982e-07  &  & 1.0726e-04 &   4.0264e-09 \\
			17	&   & 7.9956e-08  &  & 5.3631e-05   &  1.2304e-09 \\
			18	&   & 1.9989e-08  &  & 2.6816e-05   &   3.7597e-10 \\
			19	&   & 4.9973e-09 &  & 1.3408e-05   &   \\
			20	&   &  1.2493e-09 &  & 6.7041e-06   &   \\
			21	&   &3.1233e-10  &  & 3.3520e-06  &   \\
			22	&   &   &  &  1.6760e-06   &   \\
			23	&   &   &  &  8.3801e-07   &   \\
		  24	&   &   &  &  4.1900e-07   &  \\
		25	&   &   &  &  2.0950e-07   &   \\
		$ \vdots $	&   &   &  &  $ \vdots $  &   \\
	 34	&   &   &  &    4.0918e-10   &   \\
			\hline
			\hline 
		\end{tabular}
		\label{tab5}
	\end{center}
\end{table}

 \begin{figure}[H]\label{fig1}
	\begin{minipage}[H]{.46\linewidth}
		\begin{center}
			\includegraphics[width=8.2cm,height=6.2cm]{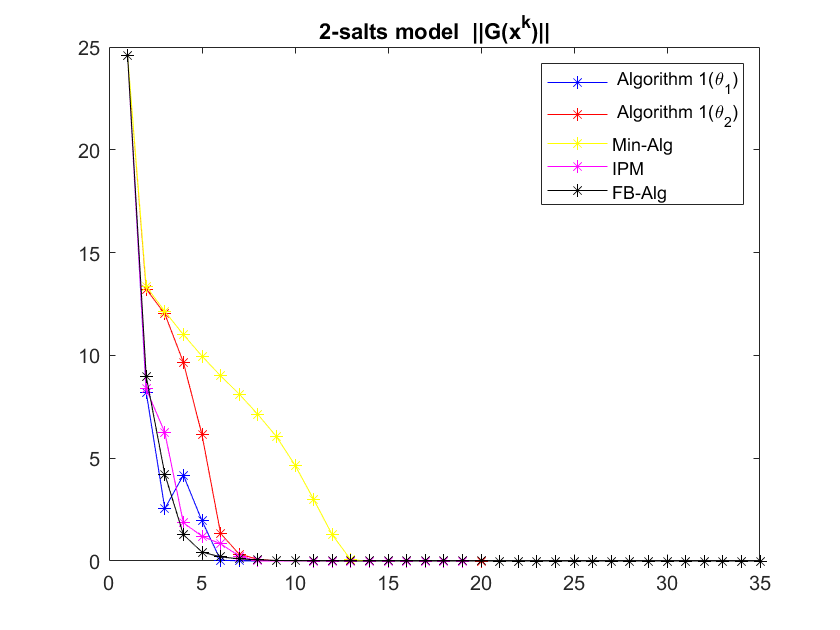}
		\end{center}
	\end{minipage} \hfill
	\begin{minipage}[H]{.46\linewidth}
		\begin{center}
			\includegraphics[width=8.2cm,height=6.2cm]{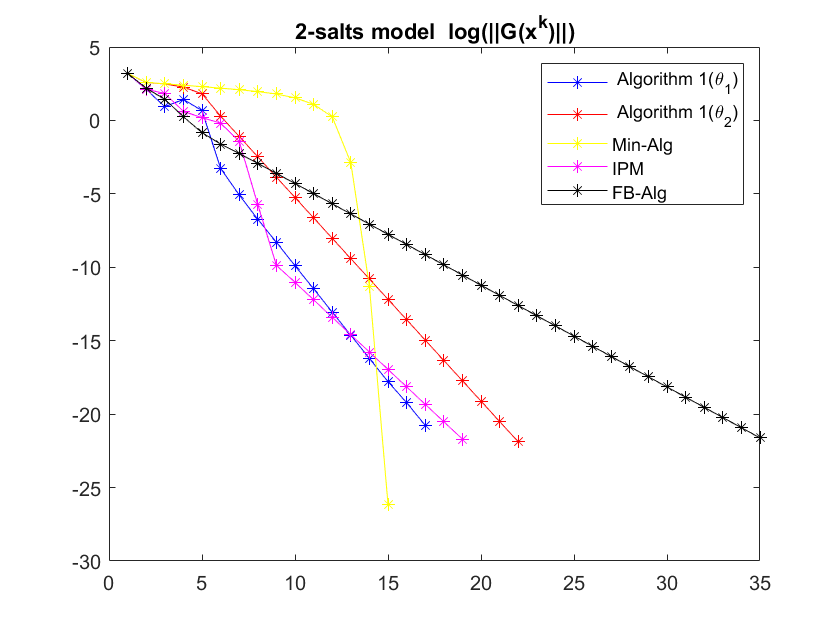}
		\end{center}
	\end{minipage}
	\caption{2-salts model of  $ \|G(x^{k})\|_{\infty} $ and $ log(\|G(x^{k}\|_{\infty}) $ .}
\end{figure}
In Figure 5 and Figure 6, all methods are quadratic convergence. The best method based on the number of iterations is semi-smooth Newton min method ($ 14 $ iters). Next is our method for $ \theta_{1}$ function with ($ 16 $ iters). Next is the classical interior-point method and our method for $ \theta_{2}$ function with ($ 18 $ iters) and  ($ 21 $ iters), respectively. The last  is the Fischer-Burmeister method with ($ 34 $ iters).
\end{example}

 \begin{example} (An ordinary differential equation) We consider the ordinary differential equation 	
\begin{equation}\label{2233}
\left\{\begin{array}{ccc}
&x^{''}(t)-|x(t)|=-2-t, \\
&x(0)=-4, ~~~~x^{'}(0)=5, \\
&t\in [0,~5].
\end{array} \right.
\end{equation}
First, we discretize the EDO equation by using the finite difference scheme. We use the second-order centred finite difference to approximate the second order derivative
\begin{equation}\label{3333}
\dfrac{x_{i-2}-2x_{i-1}+x_{i}}{h^{2}}-|x_{i}|=(-2-t)_{i}.
\end{equation}
Equation $\eqref{3333}$ was derived with equispace gridpoints $ t_{i}=ih, ~i=1, ...N.$ In order to approximate the Neumann boundary conditions we use a center difference 
\begin{equation}\label{333}
\dfrac{x_{1}-x_{-1}}{2h}=x^{'}(0)=1.
\end{equation}
Using the classical decomposition of the absolute value \cite{Lina Abdallah} we reformulate $\eqref{3333}$ as follows 
\begin{equation}\label{33344}
\left\{
\begin{array}{llllll} 
N_{1}x^{+}-N_{2}x^{-}=q,
~~~~~~~~~~~~~~~~~~~~~~~~~~~~~~~~~~~~~~~~~~~~  \\
0 \leq x^{+} \perp x^{-} \geq 0,
\end{array}
\right.
\end{equation}

where 
\begin{equation*} 
N_{1}=\frac{1}{h^2}\left(\begin{array}{ccccc}
2-h^{2}&         &        &   &     \\
-2  &1-h^{2}  &       &   &    \\
1      & \ddots  & \ddots &   &    \\
&\ddots   &  \ddots      & \ddots  &   \\
&   &1       & -2  &  1-h^{2}\\
\end{array}\right), ~~ N_{2}=\frac{1}{h^2}\left(\begin{array}{ccccc}
2+h^{2}&         &        &   &     \\
-2  &1+h^{2}  &       &   &    \\
1      & \ddots  & \ddots &   &    \\
&\ddots   &  \ddots      & \ddots  &   \\
&   &1       & -2  &  1+h^{2}\\
\end{array}\right),	
\end{equation*}
and 
$  
q=-\frac{1}{h^2}\left(\begin{array}{c}
8-10h\\
-4     \\
\vdots        \\
0\\
\end{array}\right)
-\left(\begin{array}{c}
2+h\\
2+2h \\
\vdots    \\
2+Nh\\
\end{array}\right)
$. \\

\vspace{0.2cm}
$ N_{1} $ is invertible, then the problem $\eqref{33344}$  is reduced to a standard LCP.\\
We compare the obtained solution by our methods to the predefined Runge-Kutta ode45 function in MATLAB \cite{Matlab}. The domain is $ t \in [0, 5] $, initial conditions $ x({0})=-4, x^{'}(0)=5$ and $ N=100.$
\begin{figure}[H]
	\begin{minipage}[H]{.46\linewidth}
		\begin{center}
			\includegraphics[width=8.3cm,height=6.3cm]{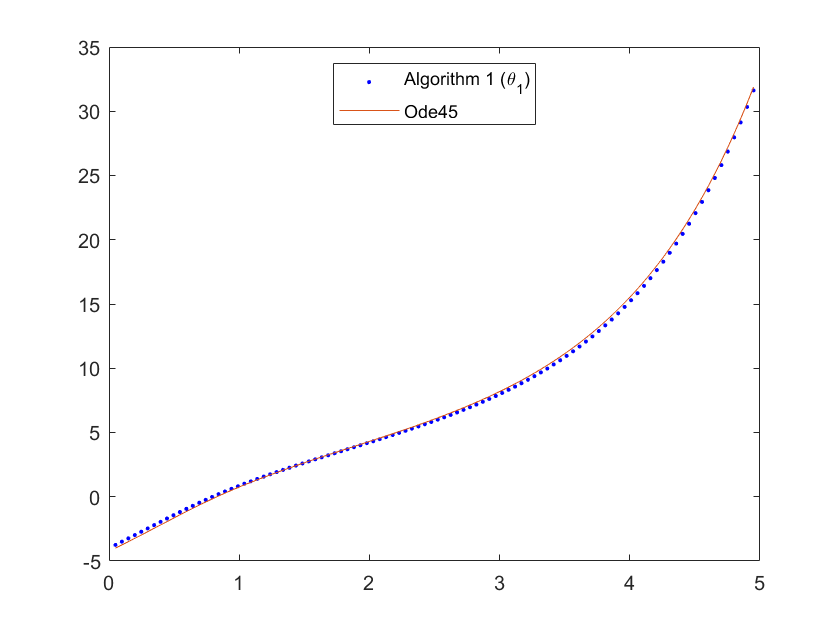}
		\end{center}
	\end{minipage} \hfill
	\begin{minipage}[H]{.46\linewidth}
		\begin{center}
			\includegraphics[width=8.3cm,height=6.3cm]{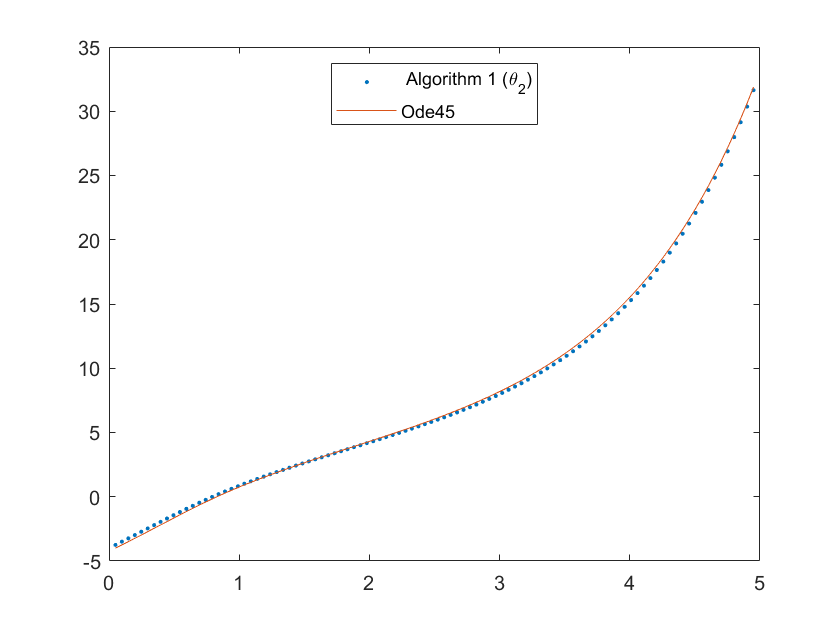}
		\end{center}
	\end{minipage}
	\caption{Numerical solution of  $\eqref{2233}$ with ode45 and both methods}
\end{figure}
Both methods solve the problem and gives consistent results.
 \end{example}

\section{Conclusion}
In this paper, we have presented a new smoothing approach for solving the nonlinear complementarity problem. For such an approach, some useful properties have been analyzed, which was employed to develop a well-defined and efficient Jacobian Newton algorithm for solving the nonlinear complementarity problem with $ \mathcal{P}_{0} $ function. We have established the global convergence and the super-linear convergence for the developed algorithm. Numerical experiments prove the efficiency of our study in the following aspects:
\vspace{0.2cm}
\begin{enumerate}
	\item It can find the solution of NCP either with less number of iteration, or with higher precision than the other.
	\item It is relatively more robust for the increasing dimension of the test problem. In particular, it seems more suitable to solve large-scale problems.
	\item It is more efficient to find the nondegenerate solution of NCP with less iteration number than the others.	
\end{enumerate}

\section{Appendix}

We give in this appendix a brief description of each test example.
\begin{enumerate}
	\item The two first examples $ \mathbf{P1} $ and $ \mathbf{P2} $ \cite{18} correspond to strongly monotone function 
	\begin{equation*}
	F(x)=(F_{1}(x),~...,~F_{n}(x))^{T} \quad \text{with} \quad F_{i}(x)=-x_{i+1}+2x_{i}-x_{i-1}+\frac{1}{3}x_{i}^{3}-b_{i}, \quad i=1,~...,~n
	\end{equation*}
	where $ x_{0}=x_{n+1}=0$ and $ b_{i}=(-1)^{i} ~(\text{resp.}~ b_{i}=\frac{(-1)^{i}}{\sqrt{i}}), \quad i=1,~...,~n,~$~ for $ \mathbf{P1} $ (resp. $ \mathbf{P2}).$
	\item $ \mathbf{P3} $ is another strongly monotone test problem from \cite{19} where $F(x)=(F_{1}(x),~...,~F_{n}(x))^{T} $ with

	\begin{equation*}
		F_{i}(x)=-x_{i+1}+2x_{i}-x_{i-1}+\arctan(x_{i})+\left(i-\frac{\pi}{2}\right), \quad i=1,~...,~n, ~~(x_{0}=x_{n+1}=0).
	\end{equation*}
	\item $ \mathbf{P4} $ and $ \mathbf{P5} $ are known as the degenerate and non-degenerate examples of Kojima-Shindo \cite{20}. $ \mathbf{P4} $ and $ \mathbf{P5} $ are respectively defined by
	
		\begin{equation*}
	F_{4}(x)=\left(\begin{array}{ccccc}
	3x_{1}^{2}+2x_{1}x_{2}+2x_{2}^{2}+x_{3}+3x_{4}-6\\
	2x_{1}^{2}+x_{1}+x_{2}^{2}+10x_{3}+2x_{4}-2\\
	3x_{1}^{2}+x_{1}x_{2}+2x_{2}^{2}+2x_{3}+9x_{4}-9\\
	x_{1}^{2}+3x_{2}^{2}+2x_{3}+3x_{4}-3\\
	\end{array}\right), \quad  \quad  
	F_{5}(x)=\left(\begin{array}{ccccc}
3x_{1}^{2}+2x_{1}x_{2}+2x_{2}^{2}+x_{3}+3x_{4}-6\\
2x_{1}^{2}+x_{1}+x_{2}^{2}+10x_{3}+2x_{4}-2\\
3x_{1}^{2}+x_{1}x_{2}+2x_{2}^{2}+2x_{3}+3x_{4}-1\\
x_{1}^{2}+3x_{2}^{2}+2x_{3}+3x_{4}-3\\
	\end{array}\right).
	\end{equation*} 
	$ \mathbf{P5} $ has a unique solution $ x^{*}=\left(\frac{\sqrt{6}}{2},~0,~0,~\frac{1}{2}\right) $ with $ F(x^{*})=\left(0,~2+\frac{\sqrt{6}}{2},~3,~0\right) $ while $ \mathbf{P4} $
 has two optimal solutions $ x^{*}=\left(\frac{\sqrt{6}}{2},~0,~0,~\frac{1}{2}\right) $  with $ F(x^{*})=\left(0,~2+\frac{\sqrt{6}}{2},~0,~0\right)$ and $ x^{**}=\left(1,~0,~3,~0\right) $ with \\ $ F(x^{*})=\left(0,~31,~0,~4\right)$. The first optimal solution of $ \mathbf{P4} $ is degenerate since $ x_{3}^{*}=F_{3}(x^{*})=0. $
\item \cite{21} In problem \eqref{1}, $ x \in \R^{7} $ and $ F(x): \R^{7}  \to \R^{7} $ is given by 

\begin{equation*}
F_{6}(x)=\left(\begin{array}{ccccc}
2x_{1}-x_{3}+x_{5}+3x_{6}-1\\
x_{2}+2x_{5}+x_{6}-x_{7}-3\\
-x_{1}+2x_{3}+x_{4}+x_{5}+2x_{6}-4x_{7}+1\\
x_{3}+x_{4}+x_{5}-x_{6}-1\\
-x_{1}-2x_{2}-x_{3}-x_{4}+5\\
-3x_{1}-x_{2}-2x_{3}+x_{4}+4\\
x_{2}+4x_{3}-1.5\\
\end{array}\right).
\end{equation*}

$ \mathbf{P6}$, has a non-degenerate solution
\begin{equation*}
x^{*}=(0.2727,~2.0909,~0,~0.54545,~0.4545,~0,~0)^{T}.
\end{equation*}

\item A complete description of $\mathbf{P7}$ and $\mathbf{P8}$ can be found in \cite{22, 23}. These two examples correspond to the Nash-Cournot test problem with $ N=5 $ and $ N=10.$ \\
Let $ x \in \R^{N}, ~Q= \sum x_{i}$ and define the functions $ C_{i}(x_{i}) $ and $p(Q)$ as follows:
\begin{equation*}
	P(Q)=5000^{\frac{1}{\gamma}}Q^{\frac{-1}{\gamma}}, \quad C_{i}(x_{i})=c_{i}x_{i}+\dfrac{b_{i}}{1+b_{i}}L_{i}^{\frac{1}{b_{i}}}x_{i}^{\frac{b_{i}+1}{b_{i}}}
\end{equation*}
The NCP-function is given by
\begin{equation*}
F_{i}(x)=C_{i}'(x_{i})-P(Q)-x_{i}p'(Q), \quad i=1, ..., N,
\end{equation*}
with $ c_{i}, L_{i}, b_{i} >0 $ and $ \gamma \geq 1.$ For our numerics, we used:
\begin{itemize}
\item  $ N=5, ~c=[10, 8, 6, 4, 2]^{T},~b=[1.2, 1.1, 1, 0.9, 0.8]^{T},~L=[5, 5, 5, 5, 5]^{T}, e=[1, 1, 1, 1, 1]^{T}$ and $ \gamma=1.1. $
\item  $ N=10, ~c=[5, 3, 8, 5, 1, 3, 7, 4, 6,3]^{T},~b=[1.2, 1, 0.9, 0.6, 1.5, 1, 0.7, 1.1, 0.95, 0.75]^{T},\\~L=[10, 10, 10, 10, 10, 10, 10, 10, 10, 10]^{T}, e=[1, 1, 1, 1, 1, 1, 1, 1, 1, 1]^{T}$ and $ \gamma=1.2. $
\end{itemize}

\end{enumerate}
Now, we report some numerics obtained by using the following projection method, (see \cite{24}).
\begin{equation*}
	x^{k+1}=\max(0,x^{k}-D^{-1}F(x^{k})), \quad k=0,~1,...
\end{equation*}
We choose $ D= \lambda I,$ where $ \lambda>0 $ is a constant and $ I $ is the $ n \times n $ identity matrix. Table 3 presents the best obtained results when varying the value of $ \lambda=(0.1,~1,~10,~20,~50,~100).$

\begin{table}[H]
	\caption{Results for the projection method }
	\begin{center}
		\begin{tabular}{ccccc}
			\hline
			\hline
			Pb & size &   Iter $  $& cpu time(s)& Opt.  \\

			\hline
			\hline 
			P1  &  10 & 71  &0.63& 2.5e-09  \\
			\hline
			&    100  & 72& 5.48 &  9.1e-11    \\
			\hline
			&  500 & 89  & 96.37&  4.7e-10   \\
			\hline
			& 1000& 83 &  224.04 & 8.8e-11   \\
			\hline
			\hline
			P2  &  10 & 72  & 1.19 & 2.2e-10     \\
			\hline
			&    100  & 80 &  5.76 & 7.1e-12   \\
			\hline
			&  500 & 91  &  112.41& 5.3e-12    \\
			\hline
			& 1000& 102 &  336.20 & 2.9e-11   \\
			\hline
			\hline
			P3  &  10 & 41 & 1.03 & 6.4e-11    \\
			\hline
			&    100  & 73 &  5.19 &  1.8e-12 \\
			\hline
			&  500 & 82  &  90.22& 5.8e-13  \\
			\hline
			& 1000& 84 &  350.06 &  2.4e-11\\
			\hline
			\hline
			P4  &  4 & 66  & 0.19& 3.1e-12    \\
			\hline
			\hline
			P5&   4  & 163 &  0.34 &  1.4e-12  \\
			\hline
			\hline
			P7&   5  & 52 &  0.22 &  6.5e-11  \\
			\hline
			\hline
			P8&   10  & 105 & 0.46 &  2.7e-12    \\
			\hline
			\hline 
		\end{tabular}
		\label{tab5}
	\end{center}
\end{table}

\subsection{Exact solution of 2-salts model}
We compute the exact solution for the problem \eqref{04081993} in case where the NCP-function is the min-function. We want to compute exact solution of 
\begin{equation}  \label{04081993}
G(x,p)= \begin{pmatrix}
T_{1}-x_{1}-p_{1} \\
T_{2}-x_{2}-P_{2}\\
x_{3}-x_{2}-x_{1}\\
\min(p_{1},~K_{1}-x_{1}x_{3})\\
\min(p_{2},~K_{2}-x_{2}x_{3})\\
\end{pmatrix}=0.
\end{equation}
When having in hand the exact solution of \eqref{04081993}, we choose the initial point in the code and also prove the existence and uniqueness of the solution of \eqref{04081993}. Here we have some conditions as $ p_{i} \geq 0,~K_{1}-x_{1}x_{3} \geq 0,~K_{2}-x_{2}x_{3}\geq 0. $ Therefore we have four cases.
\begin{enumerate}
\item If $ p_{1}>0~p_{2}>0.$ \\
In this case, since $ K_{1},~K_{2}>0 $ then from \eqref{04081993} we have $ x_{1},~ x_{2},~ x_{3}>0 $ and $ T_{1}-x_{1}=p_{1}>0 \Rightarrow  \\ 0<x_{1}<T_{1},~T_{2}-x_{2}=p_{2}>0 \Rightarrow 0<x_{2}<T_{2}.$\\
 \begin{equation*}
\left\{\begin{array}{ccc}
K_{1} &=&x_{1}x_{3} \\
K_{2} &=&x_{2}+x_{3}\\
 x_{3} &= &x_{1}+x_{2}\\
\end{array} \right. \Leftrightarrow \left\{\begin{array}{ccc}
K_{1} &=&x_{1}(x_{1}+x_{2}) \\
K_{2} &=&x_{2}(x_{1}+x_{2})\\
x_{3} &= &x_{1}+x_{2}\\
\end{array} \right. \Leftrightarrow\left\{\begin{array}{ccc}
x_{1} &=&\dfrac{K_{1}}{\sqrt{K_{1}+K_{2}}} \\
x_{2} &=&\dfrac{K_{2}}{\sqrt{K_{1}+K_{2}}}\\
x_{3} &= &\sqrt{K_{1}+K_{2}}\\
\end{array} \right.
\end{equation*}
Then $ p_{1}=T_{1}-x_{1}=T_{1}-\dfrac{K_{1}}{\sqrt{K_{1}+K_{2}}} $ and we need $ T_{1}>\dfrac{K_{1}}{\sqrt{K_{1}+K_{2}}} $ since $ p_{1}>0.$\\ The same for $ p_{2}=T_{2}-x_{2}=T_{2}-\dfrac{K_{2}}{\sqrt{K_{1}+K_{2}}} $ and the conditon $ T_{2}>\dfrac{K_{2}}{\sqrt{K_{1}+K_{2}}}$. We get the exact solution of \eqref{04081993} in this case
\begin{equation}
	(x,p)=\left(\dfrac{K_{1}}{\sqrt{K_{1}+K_{2}}},~\dfrac{K_{2}}{\sqrt{K_{1}+K_{2}}},~\sqrt{K_{1}+K_{2}},~ T_{1}-\dfrac{K_{1}}{\sqrt{K_{1}+K_{2}}},~T_{2}-\dfrac{K_{2}}{\sqrt{K_{1}+K_{2}}}\right)^{T},
\end{equation}
where $ T_{1}>\dfrac{K_{1}}{\sqrt{K_{1}+K_{2}}} $ and $ T_{2}>\dfrac{K_{2}}{\sqrt{K_{1}+K_{2}}}.$

\item If $ p_{1}=0~p_{2}>0.$ \\
Since $ p_{1}=0 $ then $ x_{1}=T_{1},~K_{1} \geq x_{1}x_{3}$ and $ x_{3}=x_{2}+x_{1} =x_{2}+T_{1}.$ Since $ p_{2}>0 $ then \\ $ K_{2}=x_{2}x_{3}=x_{2}(x_{2}+T_{1}) \Rightarrow T_{1}=\dfrac{K_{2}-x_{2}^{2}}{x_{2}} $ and we get a equation 
\begin{equation}\label{833399999333}
x_{2}=\frac{-T_{1}+\sqrt{T_{1}^{2}+4K_{2}}}{2}.
\end{equation}
Since $ T_{1} ,~T_{2} \geq 0,~K_{2}>0,~$ then $ x_{2}>0 $ and $ x_{3}=T_{2}+x_{2}>0.$  We also need $ T_{2}-x_{2}>0 $ or $ T_{2}>x_{2}.$ Then we have $ T_{2}^{2}+T_{1}T_{2}-K_{2}>0. $ The last condition to check is that $ K_{1} \geq x_{1}x_{3}=x_{1}(x_{1}+x_{2})= \\
T_{1}(T_{1}+x_{2})=\dfrac{K_{2}-x_{2}^{2}}{x_{2}}(\dfrac{K_{2}-x_{2}^{2}}{x_{2}}+x_{2}).$ Then $ x_{2} \geq \dfrac{K_{2}}{\sqrt{K_{1}+K_{2}}}$. This implies $ T_{2}>x_{2}\geq \dfrac{K_{2}}{\sqrt{K_{1}+K_{2}}}$. Then we get the solution $ x=[ T_{1},~x_{2},~T_{1}+x_{2},~0, T_{2}-x_{2} ]^{T} $ where $ x_{2} $ is in \eqref{833399999333}, $ T_{2}^{2}+T_{1}T_{2}-K_{2} >0$ \\and $ T_{2}>\dfrac{K_{2}}{\sqrt{K_{1}+K_{2}}} $.
\item If $ p_{1}>0~p_{2}=0.$ \\
Since $ p_{2}=0 $ then $ x_{2}=T_{2},~K_{2}\geq x_{2}x_{3} $ and $ x_{3}=x_{2}+x_{1}=T_{2}+x_{1}.$ Since $ p_{1}>0 $ then $ K_{1}=x_{1}x_{3}=\\x_{1}(x_{1}+T_{2}) \Rightarrow T_{2}=\dfrac{K_{1}-x_{1}^{2}}{x_{1}} $ and we get a equation
\begin{equation*}
x_{1}^{2}+T_{2}x_{1}-K_{1}=0 \Leftrightarrow x_{1}=\dfrac{-T_{2} \pm \sqrt{T_{2}^{2}+4K_{1}}}{2}.
\end{equation*}
Since we want the solution $ x $ to be nonnegative, we choose 
\begin{equation} \label{000001111}
	x_{1}=\dfrac{-T_{2} + \sqrt{T_{2}^{2}+4K_{1}}}{2}>0.
\end{equation}
If $ T_{1},~T_{2},~K_{1}>0 $ is then $ x_{1}>0 $ and $ x_{3}=T_{2}+x_{1}>0.$ We also need $ T_{1}-x_{1}>0$ or \\ $ T_{1}-\dfrac{-T_{2}+\sqrt{T_{2}^{2}+4K_{1}}}{2}>0 \Leftrightarrow T_{1}^{2}+T_{1}T_{2}-K_{1}>0$ to ensure that $ p_{1}=T_{1}-x_{1} $ is nonnegative. The last one is to check $ K_{2} \geq x_{2}x_{3}=T_{2}(T_{2}+x_{1})=\dfrac{K_{1}-x_{1}^{2}}{x_{1}}(\dfrac{K_{1}-x_{1}^{2}}{x_{1}}+x_{1})  \Leftrightarrow x_{1} \geq \dfrac{K_{1}}{\sqrt{K_{1}+K_{2}}}$ then we get \\ $ T_{1}>\dfrac{K_{1}}{\sqrt{K_{1}+K_{2}}} $. Then we get the solution $ x=[x_{1}, T_{2}, T_{2}+x_{1}, T_{2}-x_{1},0]^{T} $ where $ x_{1} $ is in \eqref{000001111} and \\ $ T_{1}^{2}+T_{1}T_{2}-K_{1}>0 $ and  $ T_{1}>\dfrac{K_{1}}{\sqrt{K_{1}+K_{2}}}$.
\item If $ p_{1}=0~p_{2}=0.$ \\
We get 
\begin{equation*}
\left\{\begin{array}{ccc}
x_{1} &=&T_{1} \\
x_{2} &=&T_{2}\\
x_{3} &= &x_{1}+x_{2}=T_{1}+T_{2}.\\
\end{array} \right.
\end{equation*}
We need some conditions that $ K_{1}-x_{1}x_{3} \geq p_{1}=0  \Rightarrow K_{1} \geq T_{1}(T_{1}+T_{2})$ and the same for \\ $ K_{2} \geq T_{2}(T_{1}+T_{2}) \geq 0.$ If $ T_{1},~T_{2} $ is nonnegative then the exact solution of \eqref{04081993} is 
\begin{equation*}
x=[T_{1},~T_{2},~T_{1}+T_{2},~0,~0]^{T},
\end{equation*}
where $ K_{1} \geq T_{1}(T_{1}+T_{2}) \geq 0 $ and $ K_{2}\geq T_{2}(T_{1}+T_{2}) \geq 0$.

\end{enumerate}


\vskip 6mm
\begin{thebibliography}{99}
	
	
	
	
	\bibitem{Lina Abdallah} L. Abdallah, M. Haddou, and T. Migot, Solving absolute value equation using complementarity and smoothing functions. \emph{J. Comput. Optim. Appl.}  $ \mathbf{327} $ (2018), 196-207.
	
	
	
	\bibitem{1}
P. T. Harker and J. S. Pang, Finite dimensional variational inequality and nonlinear complementarity problems: a survey of theory, algorithms and applications, Mathematical Programmming, vol. 48, no. 1-3, pp. 161--220, 1990.
	
	
	
	
	
	
	
	
	
	
	

	\bibitem{2}
J. J. More, Global methods for nonlinear complementarity problems, Mathematics of Operations Research, vol. 21, no. 3, pp. 589--614, 1996.
	
	
	
	
	
		\bibitem{3}
M. C. Ferris and J. S. Pang, Engineering and economic applications of complementarity problems, Society of Inian Automobile Manufacturers Review, vol. 39, no. 4, pp. 669--713, 1997.
	
	
	
\bibitem{4}
	A. Fischer, Solution of monotone complementarity problems with
	locally Lipschitz functions, Math. Programming, vol. 76, no. 2, pp. 513--532, 1997


\bibitem{5}
C. Geiger, C. Kanzow, On the solution of monotone complementarity problems, Comput. Optim. Appl. vol. 5, pp. 155--173.

	
	\bibitem{6}
	C. Kanzow, Some equation-based methods for the nonliear complementarity problem, Optim. Methods Software, vol. 3, pp. 327--340, 1994.
	
	
	\bibitem{7}
	C.-F. Ma, P.-Y. Nie, G.-P. Liang, A new smoothing equations
	approach to the nonlinear complementarity problems, J. Comput. Math, vol. 21, pp. 747--758, 2003.

	
	
	\bibitem{8}
	P.-Y. Nie, A null space approach for solving nonlinear
	complementarity problems, Acta Math. Appl. Sinica (English Ser), vol. 22, no. 1, pp. 9--20, 2006.
	
	 
\bibitem{9}
	J.S. Pang, A B-differentiable equations based, globally and
	locally quadratically convergent algorithm for nonlinear
	programming, complementarity, and variational inequality problems, Math. Programming 51, pp. 101--131, 1991.

	
	

	
		\bibitem{10}
J. S. Pang, S. A. Gabriel, NE-SQP: A robust algorithm for nonlinear complementarity problem, Math. Programming 60, 1993.

	
	
	
	\bibitem{11}
A. Fischer, A special Newton-type optimizaton method, Optimization, vol. 24, no. 3-4, pp. 269--284, 1992.
	
	
	
	
	
	
	\bibitem{12}
	 C. Kanzow, N.Yamashita and M.Fukushima. New NCP functions and
	their properties. Journal of Optimization Theory and Applications,
	vol 94, (1997) pp. 115-135.
	
	
	
	
	\bibitem{13}
	L. Zhang, J. Han, Z. Huang, Superlinear quadratic one step smoothing Newton  method for P0 NCP, Acta Math. Sinica 26 (2) (2005) 117--128. 
	
	
		\bibitem{14}
C. Chen and O. L. Mangasarian, Aclass of smoothing functions for nonlinear complementarity problems, Journal of Computational and Applied Mathematics, vol. 80, pp. 105--126, 1997.

		\bibitem{15}
C. Kanzow and H. Pieper, Jacobian smoothing methods for nonlinear complementarity problems, SIAM Journal on Optimization, vol. 9, no. 2, pp. 342--373, 1999.
	
	
	\bibitem{16}
N. Krejic and S. Rapaji, Globally convergent Jacobian smoothing inexact Newton methods for NCP, Computational Optimization and Applications, vol. 41, pp. 243--261, 2008.
	
	

	\bibitem{17}
L. Qi and D. H. Li, A smoothing Newton method for nonlinear complementarity problems, Advanced Modeling and Optimization, vol. 13, no. 2, pp. 141--152, 2011.
	\bibitem{18}
	C. Huang and S. Wang, A power penalty approach to a Nonlinear Complementarity Problem, Operations research letters, vol. 38, no. 1, pp. 72--76, 2010. 
	
	
		\bibitem{19}
	L. Dong-hui and Z. Jin-ping, A penalty technique for nonlinear problems, Journal of Computational Mathematics, vol. 16, no. 1, pp. 40--50, pp. 1998.
	
	
	\bibitem{20}
M. Kojima and S. Shindo, Extensions of Newton and quasi-Newton methods to systems of PC1 equations, J. Oper. Res. Soc. Jpn, vol. 29, pp. 352--374, 1986.
	
		\bibitem{21}
	W. Zhong, Y. Min and W. Chang, A partially smoothing Jacobian method for nonlinear complementarity problems with P0 function, Journal of Computational and Applied Mathematics, vol. 286, pp. 158--171, 2015.
	
		\bibitem{22}
$  http://dm.unife.it/pn2o/software/Extragradient/test-problems.html$
	
	\bibitem{23}
	
P. T. Harker, Accelerating the convergence of the diagonalization and projection algorithms for finite-dimensional variational inequalities, Mathematical Programming, vol. 48, pp. 29--59, 1990.
	

\bibitem{24}

F. Facchinei and J. S. Pang, Finite-dimensional Variational Inequalities and Complementarity Problem, Springer Series in Operations Research, Springer-Verlag, New York, vol. 1-2, 2003.
	
	\bibitem{25}
	
M. Tangi, Jocelyne Erhel. Analyse methématique de modèles géochimiques, Rapport de recherche, INRIA Rennes, équope SAGE, 2014.
	\bibitem{26}
	E. D. Dolan and J. J. Moré. Benchmarking optimization software with performance profiles. Mathematical Programming, vol. 91, pp. 201--213, 2002.
	
	
	
	
	
	
	

		\bibitem{Bonnans}
	F. Bonnans, {\em Optimisation continue: cours et problèmes corrigés, Mathématiques appliquées pour le Master}, Dunod.  (2006).

	

	\bibitem{El Ghami} M. El Ghami,  New primal-dual interior-point methods based on kernel functions. PhD thesis, 2005.
	
	
		\bibitem{FB} X. Chen, B. Chen and C. Kanzow, A penalized Fischer-Burmeister NCP-function: theoretical investigation and numerical results, Int. fur Angewandte Mathmatik der Univ. (1997).
	

	
	
	
	
	\bibitem{M. Haddou and P. Maheux} M. Haddou and P. Maheux, Smoothing methods for nonlinear complementarity
	problems. \emph{J. Optim. Theory. Appl.} $ \mathbf{3} $ (2014) 711-729.
	\bibitem{HADDOU}  M. Haddou, A New Class of Smoothing Methods for Mathematical Programs With Equilibrium Constraints. \emph{Pacific Journal of Optimization.}  $ \mathbf{5} $  (2009) 87-95.  
	
	\bibitem{27}
	A. Auslender, R. Cominetti, M. Haddou, Asymptotic analysis for penalty and barrier methods in convex and linear
	programming, Math. Oper. Res. 22 (1997) 43–62. doi:10.1287/moor.22.1.43.
	
	
	
	
	\bibitem{28}
	M. Bergounioux, M. Haddou, A new relaxation method for a discrete image restoration problem, J. Convex Anal.
	17 (2010) 861–883.
	\bibitem{29}
	M. Haddou, T. Migot, J. Omer, A generalized direction in interior point method for monotone linear complementarity
	problems, Optim. Lett. (2018). doi:10.1007/s11590-018-1241-2.
	\bibitem{30}
	T. Migot, Contributions aux méthodes numériques pour les problèmes de complémentarité et problèmes d’optimisation sous contraintes de complèmentarité, Ph.D. thesis, INSA Rennes, 2017.
	
	
	\bibitem{Yamasshita}
N. Yamashita, M. Fukushima, Modified Newton methods for solving a semismooth reformulation of monotone complementarity problems. Math. Program. 76, 469-491 (1997)
	
	
	

	

	\bibitem{Matlab} MATLAB. version R2020. Natick, Massachusetts: The Math Works Inc, 2019.


	
	
	
	
	
	\bibitem{Tran} D. Thach, I.B. Gharbia, M. Haddou and Q.H. Tran, A new approach for solving nonlinear algebraic systems with complementarity conditions application to compositional multiphase equilibrium, Preprint sumbmitted to some Elsevier journal 09/2020, in press.
	

	
\end{thebibliography}
\end{document}